\documentclass{birkmult}

\usepackage{amsmath,amsthm,amssymb,mathrsfs,amscd,amsthm,amscd,
amsfonts,graphicx,mathtools, bbm}
\usepackage[cmtip,all]{xy} 	
\usepackage{color}
\usepackage{hyperref}
\hypersetup{
	backref,
	pagebackref,
	colorlinks=true,
	linkcolor=blue,
	filecolor=magenta,      
	urlcolor=cyan,
}

\newcommand{\ko}{{\mathcal O}}

\def\R{\mathbbm{R}}
\def\C{\mathbbm{C}}
\def\Z{\mathbbm{Z}}
\def\Q{\mathbbm{Q}}

\def\fm{\frak{m}} 
\def\fn{\frak{n}}
\def\fp{\frak{p}}
\def\fq{\frak{q}}

\def\x{\langle x \rangle}

\DeclareMathOperator{\Spec}{Spec}
\DeclareMathOperator{\Max}{Max}
\DeclareMathOperator{\Ann}{Ann}
\DeclareMathOperator{\Sing}{Sing}
\DeclareMathOperator{\Sm}{Sm}
\DeclareMathOperator{\NSm}{NSm}
\DeclareMathOperator{\Reg}{Reg}
\DeclareMathOperator{\Supp}{Supp}

\DeclareMathOperator{\Quot}{Quot}

\DeclareMathOperator{\rank}{rank}
\DeclareMathOperator{\rad}{rad}
\DeclareMathOperator{\chara}{char}
\DeclareMathOperator{\coker}{coker}

\DeclareMathOperator{\mng}{mng}

\newtheorem{theorem}{Theorem}
\newtheorem{lemma}[theorem]{Lemma}
\newtheorem{corollary}[theorem]{Corollary}

\newtheorem{proposition}[theorem]{Proposition}
\newtheorem{definition}[theorem]{Definition}
\newtheorem{notation}[theorem]{Notation}
\newtheorem{example}[theorem]{Example}
\newtheorem{remark}[theorem]{Remark}

\begin{document}

%
%

\begin{center}
{\bf \large Semicontinuity for families of power series II}\\
\medskip
{\bf Dedicated to L\^{e} D\~{u}ng Tr{á}ng}
\bigskip

Gert-Martin Greuel and Gerhard Pfister
\end{center}

\email{gm.greuel@gmail.com, greuel@rptu.de}
\email{pfister@mathematik.uni-kl.de, pfister@rptu.de}
 
\bigskip
\noindent {\bf Abstract}\\
The present article is a generalization of our paper \cite{GP21}, together with new applications.
The original problem  came up in connection with the classification of singularities in positive characteristic. 
Then it is important that certain invariants 
cannot increase if we deform a given singularity. In \cite{GP21} we proved this for the completed fiber dimension under deformations with section. In general the nearby fiber contains however several singularities and in this article we prove that also the sum of the completed fiber dimensions behaves semicontinuous.
This is well known for analytic 
or algebraic families. However, for families of power series the problem is more difficult, since the the modules defining the invariants
are quasi-finite but not finite over the base space.

In fact, in general the usual fibre dimension is not semicontinuous. 
However, if we pass to the completed fibres in a family of rings or modules over arbitrary Noetherian rings, we can prove that  their dimension  is semicontinuous. The proof is different from that given  in \cite{GP21} and, though  more general, even easier.

Finally we apply this to prove the semicontinuity of several singularity invariants, such as the Milnor number and the Tjurina number in families of hypersurfaces and complete intersections. We end with a problem concerning the Milnor number of an isolated complete intersection singularity.  


\renewcommand{\contentsname}{Table of Contents}
\tableofcontents

\section*{Introduction}
The semicontinuity of (numerical) invariants is an important property  in the study of varieties and singularities, when these vary in a family over some base space. 
One possible application is the classification of singularities and their modality by investigating adjacent singularities that appear in nearby fibers. Semicontinuity then means that the singularities in the nearby fibers are "less complicated". To study the nearby fibers it is not sufficient  to consider families over complete local rings but one has to consider families of power series parametrized by some algebraic variety in the neighbourhood of a given point. In this paper we consider families of formal power series over the spectrum of an arbitrary Noetherian ring (of any characteristic) with fibers over closed as well non-closed points.\\
Given a morphism $A\to R$ of Noetherian rings, the usual fiber 
$R(\fp) = R\otimes_Ak(\fp)$, $\fp \in \Spec A$, $k(\fp)$ the residue field of $A_\fp$, is not the right object for our applications. We need the completed fiber 
$\hat R(\fp)=R(\fp)^{\wedge,I}$, where ${}^{\wedge,I}$ denotes the $I$-adic completion, $I\subset R$ an ideal such that $R/I$ is finite over $A$. Then $\hat R(\fp)$ is a semi-local ring (this is not the case with $R(\fp)$) and we study the singularities $\hat R(\fp)$, when $\fp$ varies in $\Spec A$. The treatment of fibers over non-closed points is not only theoretically interesting but has unexpected computational applications as is shown in \cite{GPS22}.\\
In the first chapter we introduce the completed tensor product in order to the individual completed fibers $\hat R(\fp)$ in a family over $A$  and study its properties. In the second chapter we study the completed fiber as   semi-local ring and in the third chapter we prove our main semicontinuity theorem (Theorem \ref{thm.scont}). In chapter four we clarify the relation between smoothness and (absolute) regularity. In the last chapter we apply our main theorem to prove the semicontinuity of the Milnor number and the Tjurina number in families of hypersurface and complete intersection singularities. Other invariants can be treated in a similar way, but we do not pursue this here (however, not all singularity invariants are semicontinuous).\\
The proof of our main semicontinuity theorem is different and even easier than the one given in \cite{GP21} in the special case, where $R = A[[x]]$ and where $I=\x$ is the ideal of the trivial section. In \cite{GP21} we used the Henselian tensor product to prove the semicontinuity for modules given by formal power series that admit an algebraic presentation. Due to our new approach here, we generalize and improve the results of  \cite{GP21} by omitting this restriction.\\ 

{\bf Acknowledgments:} This article is dedicated to the memory of L\^{e} D\~{u}ng Tr{á}ng, with whom we shared a long-standing friendship and a fruitful mathematically collaboration, which has also resulted in joint publications.
We would like to
thank the Mathematical Institute of the University of Freiburg for its hospitality and the
RPTU Kaiserslautern for its support.

\section{The completed tensor product}\label{sec.qms}
\medskip

In this section we consider families of completed rings and modules, parametrized by the spectrum of a Noetherian ring. We recall first some basic facts about completions.

Let $P$ be any ring and $I\subset P$ an ideal.  
The inverse limit 
$$P^\wedge:= P^{\wedge,I}:=\lim\limits_{\longleftarrow}(P/I^m)$$
of the inverse system $\{P/I^{m+1}\twoheadrightarrow P/I^{m}\}$
is called  the {\em completion of $P$ with respect to $I$} or the {\em $I$-adic completion}
of $P$. An element of $P^\wedge$ is a sequence $(f_1,f_2,...)$, $f_m \in P/I^m$, such that $f_{m+1}/I^m = f_m$.
We have a natural ring map $P\to P^{\wedge}$, given by $f\mapsto (f/I,f/I^2,...)$.
For example, if  $P=A[x], x=(x_1,...,x_n)$, $A$ a ring, and $I = \langle x_1,\ldots,x_n\rangle$, then  $P^{\wedge, I}=A[[x]]$ is the formal power series ring over $A$. 
 
This generalizes immediately to a $P$-module $M$.\footnote{We assume all rings to be associative, commutative and with unit and any module unitary.} 
 The $I$-adic completion of $M$ is $M^\wedge := M^{\wedge,{I}}:= \lim\limits_{\longleftarrow}(M/I^mM)$. An element $x\in M^\wedge$ is given as a sequence 
$(x_1,x_2,...)$, $x_m \in M/I^mM$, such that $x_{m+1}/I^mM = x_m$. We have a natural map $M \to M^{\wedge}$, $x\mapsto (x/IM, x/I^2M,...)$,  which has $\bigcap I^mM$ as kernel.
 $P$ (resp. $M$) is called {\em ($I$-adically) complete} if  $P \to P^{\wedge}$ (resp. $M \to M^{\wedge}$) is an isomorphism. 

For any ring map $P \to Q$ we often write $Q^{\wedge, I}$ or $Q^\wedge$ instead of $Q^{\wedge, IQ}$ and we have a natural ring map
 $P^{\wedge} \to Q^{\wedge}$. A $Q$-module $N$ is  also a $P$-module and we have
$N^{\wedge,{IQ}}= \lim\limits_{\longleftarrow}(N/I^mN) = N^{\wedge,I}$. Then
$N^{\wedge}$ is in a natural way a $P^\wedge$-module and a $Q^\wedge$-module. 
\\

  The following properties of the completion are well known.
  
\begin{lemma}\label{lem.compl}
Let $P$ be a Noetherian ring, $I\subset P$ an ideal and 
$M$ a $P$-module. Let
${}^{\wedge}$ denote the $I$-adic completion.
\begin{enumerate}
\item $M^{\wedge}$ depends only on the radical of $I$.
\item $M^{\wedge}$ is $I$-adically complete.
\item $I^m M^\wedge = (I^m M)^\wedge$ and $M/I^m M=M^{\wedge}/I^m M^{\wedge}$ for all $m \ge 1$.
\item If $P$ is $I$-adically complete then $I\subset \rad (P)$.\footnote{$\rad (P)$ denotes the Jacobson radical of $P$.}\\
In particular, $IP^\wedge \subset \rad(P^\wedge)$.
\item $P^{\wedge}$ is Noetherian and  $P\to P^{\wedge}$ is flat.
\item If  $M$ is a finite $P$-module, then $M^\wedge = M\otimes_P P^\wedge$. \\
In particular, if $J\subset P$ is an ideal, then  $(P/J)^\wedge = P^\wedge / JP^\wedge$.
\item The completion is an exact functor on the caetgory of finite $P$-modules.
\item  Let $J\subset P$ be an ideal. If $P$ is
$I$-adically complete and $P/I$ is $J$-adically complete, then $P$ is $J$-adically complete.
\item Let  $M$ be a finite $P$-module. The kernel 
$\cap_{n=1}^\infty I^nM$ of $M \to M^\wedge$ consists of those $x \in M$ annihilated by some element of $1 + I$. In particular, 
since $(M^\wedge)^\wedge = M^\wedge$,
$\cap_{n=1}^\infty I^nM^\wedge =0$.
\item Let  $I\subset \rad(P)$. If $M$ is finite over $P$ then $M\to M^\wedge$ is injective. In particular,  $P\to P^\wedge$  is injective
and $ P^\wedge$ is faithfully flat over $P$.
Furthermore, if  $IM =M$ then $M=0$ (Nakayama's lemma).
\item Let $J\subset P$ be an ideal and $M$ a finite $P$-module, then $(M^{\wedge,I})^{\wedge, J} = M^{\wedge, I+J}$. \item If $I=\fm$ is a maximal ideal then $P^{\wedge} = (P_\fm)^{\wedge,\fm}$ is a local ring with maximal ideal $\fm P^{\wedge}$ and residue field $P^{\wedge}/\fm P^{\wedge} = P/\fm.$ If $P$ is regular then $P^\wedge$ is regular.
\item If $(P,\fm)$ is local, then $P$ is regular iff its $\fm$-adic completion $P^{\wedge}$ is regular.
\end{enumerate} 
\end{lemma}

\begin{proof}
1. follows from \cite[Lemma 10.95.9, tag 0319]{Sta19}, 2. and 3. from \cite[Lemma 10.95.3, tag 05GG]{Sta19}, 4. from \cite[Lemma 10.95.6, tag 05GI]{Sta19}. 5. follows from \cite[Lemma 10.96.6, tag 0316 and Lemma 10.96.2, tag 00MB]{Sta19}, 6. and 7.  from \cite[Lemma 10.96.1, tag 00MA]{Sta19} and 8. from 
\cite[Lemma 10.96.10, tag 0DYC]{Sta19}.
9. follows from \cite[Theorem 10.17]{AM69},  and
10. from  \cite[Corollary 10.19]{AM69}
and \cite[Theorem 8.14 and Theorem 2.2]{Mat86}.
For 11. see  \cite[Section 10, Exercise 5]{AM69}.

For 12. let $\fm$ be maximal. Then $P_\fm/(\fm P_\fm)^n =(P/\fm^n)_\fm
= P/\fm^n$ since $P/\fm^n$ is already local with maximal ideal 
$\fm/\fm^n$. Hence $(P_\fm)^\wedge =P^\wedge$ is local and the residue fields satisfy 
$(P_\fm)^\wedge/\fm(P_\fm)^\wedge =P_\fm/\fm P_\fm =  \Quot (P/\fm) = P/\fm = P^\wedge/\fm P^\wedge$. 
The regularity of $P^\wedge$ follows from 13..

For 13. recall that a Noetherian local ring $(P,\fm)$ with residue field $k$ is regular if its embedding dimension $\dim_{k}\fm/\fm^2$ equals its dimension $\dim P$  (cf. Definition \ref{def.sing}). It follows that $P$ is regular iﬀ its completion is regular, since the dimension and embedding dimension remain the same after
completion.
\end{proof}

\begin{remark} \label{rm.compl}
{\em Although we are mainly interested in Noetherian rings, we like to mention that some of the above properties 1. - 12. hold more generally. Let $P$ be any ring and $M$ a $P$-module.\\
1. holds without any further assumption (\cite[Lemma 10.95.9, tag 0319]{Sta19}).\\
2. and 3. hold if the ideal $I\subset P$ is finitely generated ( \cite[Lemma 10.95.3, tag 05GG]{Sta19}).  4. generalizes to $IR^\wedge \subset \rad(P^\wedge)$ for any ideal $I$ by \cite[Lemma 10.95.6, tag 05GI]{Sta19}.\\
The following lemma is wrong if $P$ is not Noetherian.  The canonical map $M\otimes_PN^\wedge\to (M\otimes_PN)^\wedge$ is always surjective but in general not injective. For a counterexample see \cite[tag 05JD]{Sta19}.}
\end{remark}

\begin{lemma}\label{lem.lim}
Let $P$ be Noetherian, $I\subset P$ an ideal and ${}^\wedge$ the $I$-adic completion. Then
$$(M\otimes_PN)^\wedge = M^\wedge \otimes_PN = M\otimes_PN^\wedge$$
for finitely generated $P$-modules $M$ and $N$.
\end{lemma}
\begin{proof}
By Lemma \ref{lem.compl}.6., since  $M\otimes_PN$ is finite over $P$,
$(M\otimes_PN)^\wedge 
= (M\otimes_PN)\otimes_PP^\wedge = (M\otimes_PP^\wedge)\otimes_PN=  M\otimes_P(N\otimes_PP^\wedge)$. Now apply 
 Lemma \ref{lem.compl}.6. again.
\end{proof}

  We are interested in families of completed rings or modules over the spectrum of an arbitrary Noetherian ring $A$.  Before we consider families with completed fibers, we consider first the case of ordinary fibers.

 Let $A  \to R$ be a ring map and $M$ an $R$-module. For a prime ideal  $\fp$ of $A$ let $k(\fp) := A_\fp/\fp A_\fp = \Quot(A/\fp)$  be the residue field of $\fp$. 
We set
$$
\begin{array}{lll}
R_\fp &:= R \otimes_{A} A_\fp, & R(\fp) := R \otimes_{A} k(\fp),\\
M_\fp & := M \otimes_{A} A_\fp, & M(\fp):= M \otimes_{A} k(\fp), 
\end{array}
$$
and call $M_\fp$  the {\em stalk} of $M$ over $\fp$ and 
$M(\fp)$ the {\em fiber} of $M$ over $\fp$. $M_\fp$ is the localization of the $A$-module $M$ at $\fp$ (\cite[Proposition 3.5]{AM69}),  and $M(\fp) 
= M_\fp \otimes_{A_\fp} A_\fp/\fp A_\fp
=M_\fp/\fp M_\fp$.
$M_\fp$ is an  $R_\fp $-module and 
$M(\fp)$ an $R(\fp) $-module, which is  also a $k(\fp) $-vector space via
$k(\fp) \to R(\fp)$. We call 
$$d_{\fp} (M):= \dim_{k(\fp)} M(\fp)$$
{\em the fiber dimension of $M$ over} $\fp$. We remark that $d_\fp(M) <\infty$ if $M$  is finitely generated as $A$-module.  In this case the semicontinuity of $d_{\fp} (M)$ holds and is well known (see lemma \ref{lem.free}). We are however interested in the case when $M$ is finite over $R$ but not over $A$ (see e.g. Example \ref{ex.notsemcont}).

\begin{remark} \em{In our applications the fiber $M(\fp)$ is not the right object to consider if $\fp$ is not a maximal ideal. For example, let $A=K[t]$, $K$ a field, $M=R=A[[x]]=A[x]^{\wedge,\x}$ and $\fp=\langle 0\rangle$. Then the fiber $R(\fp) = A[[x]](\langle 0\rangle)=(K[t])[[x]] \otimes_{K[t]}K(t)$ is  neither local nor complete, while we are interested in the complete local ring $A(\langle 0\rangle)[[x]]= K(t)[[x]]$, which is the $\x$-adic completion of $A[[x]](\langle 0\rangle)$. That is, the right object to consider is not the fiber $M(\fp)$ but the completed fiber
(in this case $M(\fp)^{\wedge,\x}$), which we define now.}
\end{remark}

\begin{definition} \label{def.cf} Let  $A  \to R$ ba a ring map, $I\subset R$ an ideal, $M$ an $R$-module,
and $\fp$ a prime ideal   of $A$. We define
$$
\begin{array}{lll}
\hat R_\fp &:= (R_\fp)^{\wedge,I},  & \hat R(\fp) := R(\fp)^{\wedge,I}\\
\hat M_\fp & :=(M_\fp)^{\wedge,I} &
\hat M(\fp):= M(\fp)^{\wedge,I}  
\end{array}
$$
and call $\hat M_\fp$  the {\em completed stalk} of $M$ over $\fp$ and 
$\hat M(\fp)$ the {\em completed fiber} of $M$ over $\fp$.         
$\hat M_\fp$ is an  $\hat R_\fp$-module and 
$\hat M(\fp)$ an $\hat R(\fp)$-module, which is  also a $k(\fp) $-vector space. We call 
$$\hat d_{\fp} (M):= \dim_{k(\fp)} \hat M(\fp)$$
{\em the completed fiber dimension of $M$ over} $\fp$. 
\end{definition}

\bigskip

To put the individual completions $\hat M(\fp)$ into a family over $\Spec A$, we introduce the completed tensor product.

\begin{definition}\label{def.ctp}
Let  $A \to R$ be a ring map, $I\subset R$ an ideal, and  $B$ an $A$-algebra . 
We define  the {\em ($I$-adic) completed tensor product} of $R$  and $B$ over $A$ as the ring
$$R \hat\otimes_A B :=  \lim\limits_{\longleftarrow}\big((R/I^mR) \otimes_A B \big).$$
If  $M$ is an $R$-module  and $N$ a $B$-module, we define the {\em ($I$-adic) completed tensor product} of  $M$  and $N$ over $A$ as the $R \hat\otimes_A B$-module 
$$M \hat\otimes_A N :=  \lim\limits_{\longleftarrow}\big((M/I^m M)\otimes_A N\big).$$
\end{definition}

In the following we collect the most important properties of the completed tensor product. 

\begin{proposition}\label{prop.pctp} Let  $A \to R$ be a ring map,
$I\subset R$ an ideal, $B$ an $A$-algebra, $M$ an $R$-module, and $N$ a $B$-module.
The $I$-adic completed tensor product has the following properties.
\begin{enumerate}
\item
$R \hat\otimes_A B = (R \otimes_A B)^{\wedge}$,
$M \hat\otimes_A N = (M \otimes_A N)^{\wedge}$.\\
 If $I$ is finitely generated, then 
$M \hat\otimes_A N= M^\wedge  \hat\otimes_A N$.
In particular,
$$ M\hat \otimes_A A_\fp =(M^\wedge)\hat\otimes_A A_\fp, \quad  M\hat \otimes_A k(\fp) =(M^\wedge)\hat\otimes_A k(\fp).$$
\item 
Let $A$, $B$ and $R$ be Noetherian. If $M$ is finitely generated over $R$ and $N$ is finitely genrated over $B$, then
$$
\begin{array}{lll}
M \hat\otimes_A N &\cong (M  \otimes_A N) \otimes_{R\otimes_A B}(R\otimes_A B)^{\wedge,IR \otimes_A B}\\
&= (M  \otimes_A N) \otimes_{R\otimes_A B}(R\hat\otimes_A B).
\end{array}
$$  
\item Let $A$ and $R$ be Noetherian and $I \subset \rad(R)$. 
Then the canonical map
 $M \otimes_A N \to M \hat\otimes_A N$  is injective if $M$ is finite over $R$ and $N$ finite over $A$.
\medskip
\item If $I^m M =0$ for some integer $m$  then
$M \hat\otimes_A N = M \otimes_A N$ for every $A$-module $N$. 
 \end{enumerate}
\end{proposition} 

\begin{example}\label{ex.pctp}{\em
In particular, let $A$, $B$ be Noetherian, $R=A[x]/J$, $J\subset A[x]$ an ideal, and  $I\subset R$ an ideal. Then  statement 1. and Lemma \ref{lem.compl}.6. imply ($^\wedge$ the $I$-adic completion):
$$(A[x]/J)\hat \otimes_AB =((A[x]/J) \otimes_AB)^\wedge=(B[x]/JB[x])^{\wedge}=B[x]^\wedge/JB[x]^{\wedge}.$$
E.g.  for  $R=A[x]$ and $B=k(\fp)$ we get $A[x]\hat \otimes_A k(\fp)=k(\fp)[x]^{\wedge}$ and if $I = \langle x_1,\ldots,x_n\rangle$,
then $A[[x]]\hat \otimes_AB=B[[x]]$ and $A[[x]]\hat \otimes_Ak(\fp)=k(\fp)[[x]]$.}
\end{example}

\begin{proof}[Proof of Proposition \ref{prop.pctp}]
1. The equality  $M \hat\otimes_A N = (M \otimes_AN)^{\wedge}$ follows from  
$$(M/I^mM)\otimes_AN = (M\otimes_A N)/I^m ({M\otimes_A N}).$$
Moreover, since $M/I^m M = M^\wedge/I^m M^\wedge$ by Lemma \ref{lem.compl} and Remark \ref{rm.compl}
 we have also
$M^\wedge  \hat\otimes_A N =\lim\limits_{\longleftarrow}\big((M^\wedge/I^m M^\wedge)\otimes_A N\big) = M\hat\otimes_A N.$

2. Since $M$ is  finitely generated over $R$ and $N$ is finitely generated over $B$, 
$M\otimes_A N$ is finitely generated over the Noetherian ring $R\otimes_A B$.
Using item 1,  
Proposition \ref{lem.compl}.6. gives the isomorphism.

3.  If
$M$ is finitely generated over $R$ and $N$ is finitely generated over $A$ then $M\otimes_A N$ is finitely generated over $R =R\otimes_A A$. 
The injectivity follows from 1., from 2. (with $B=A$), and from Lemma \ref{lem.compl}.10.

4.  Follows from the  definition of the completed tensor product.
 \end{proof}

 In the following corollaries \ref{cor.cpl} -- \ref{cor.fitt}  we use the notations  of Proposition \ref{prop.pctp} and assume that 
$A$, $B$, $R$ are Noetherian rings and  $M$ (resp. $N$) are finite $R$ (resp. $A$) modules. 
 \begin{corollary}\label{cor.cpl} 
We have:
 \begin{enumerate}
\item[(i)] 
 $R\hat\otimes_A A = R^\wedge$, $M\hat\otimes_A A =M^{\wedge}$.
 \medskip
 \item[(ii)]  $\hat M_\fp =M\hat\otimes_A A_\fp$ and $\hat M(\fp) = M\hat\otimes_A k(\fp)$ for $\fp \in \Spec A$. 
\medskip
 \item [(iii)] If $S$ be multiplicatively closed in $A$, then 
 $M\hat\otimes_A (S^{-1}A) = (S^{-1}M)^{\wedge,IS^{-1}R}$.
 \medskip
 \item[(iv)] If $R$ is $I$-adically complete, then $M$ is I-adically complete. 
\end{enumerate}
\end{corollary}
\begin{proof}
Proposition \ref{prop.pctp}.1  implies (i) and (ii) with $B=A$ and $N=A_\fp$ resp. $N=k(\fp) = A_\fp/\fp$ by definition of 
$\hat M_\fp$ and $\hat M(\fp)$, while  
(iii)  follows from  Proposition \ref{prop.pctp}.2. with $N=B=S^{-1}A$, and (iv) follows from  Lemma \ref{lem.compl}.6..
  \end{proof}

\begin{corollary}\label{cor.rex}
The completed tensor product is right-exact on the category of finitely generated modules. That is, let  
$$
\begin{array}{*{20}l}
M' \to M \to  M'' \to 0 ,  \text{resp.}\\
N' \to N \to N'' \to 0
\end{array}
$$
be exact sequences with $M', M, M''$ finitely generated  $R$-modules and $N', N, N''$ finitely generated   $B$-modules. Then the sequences of finitely generated  $R \hat\otimes_A B$-modules
$$
\begin{array}{*{20}l}
M' \hat\otimes_A N\to M\hat\otimes_A N\to M'' \hat\otimes_A N\to 0,\text{resp.}\\
M \hat\otimes_A N'\to M\hat\otimes_A N\to M\hat\otimes_A N''\to 0\\
\end{array}
$$
are exact. 
\end{corollary}

\begin{proof}
The sequences 
 $ M'\otimes_A N \to M\otimes_A N\to M'' \otimes_A N\to 0$ and 
  $ M\otimes_A N' \to M\otimes_A N\to M\otimes_A N''\to 0$ 
are exact. Now tensor these sequences with $R\hat\otimes_A B$ over 
$R\otimes_A B$  and apply Proposition \ref{prop.pctp}.2. 
 \end{proof}

 \begin{corollary}\label{cor.maxA} We have:
\begin{enumerate}
\item [(1)]  $\hat M(\fp)= \hat M_\fp /\fp \hat M_\fp$. 
 In particular
$\hat R(\fp)= \hat R_\fp /\fp \hat R_\fp$.\\
If $I \subset \fp R_\fp$ then $\hat R(\fp) =  R(\fp)$.
\item [(2)]  Let $\fp$ be a maximal ideal.  Then
 $$ \hat M(\fp) =  M(\fp)^\wedge = M^\wedge (\fp) =M^\wedge/\fp M^\wedge.$$ In particular, if $R$ is complete, then  $ \hat M(\fp) = M(\fp)$
and hence  $\hat d_\fp (M) = d_\fp (M). $
\end{enumerate}
\end{corollary}

\begin{proof}
We apply $M \hat\otimes_A -$ to the exact sequence  $\fp A_\fp \to A_\fp \to k(\fp) \to 0$  of finitely generated  $A_\fp$-modules ($A_\fp$ is Noetherian since $A$ is). 
Corollary \ref{cor.rex} (second sequence with $N=B=A_\fp$) and Corollary \ref{cor.cpl} (ii) gives the first line of (1).
If $I \subset \fp$ then $IR_\fp \subset \fp R_\fp$ and 
$\hat R_\fp /\fp \hat R_\fp = \hat R_\fp /\fp \hat R_\fp + I \hat R_\fp = (\hat R_\fp / I \hat R_\fp)/\fp \hat R_\fp (\hat R_\fp / I \hat R_\fp)= R_\fp / I R_\fp +\fp R_\fp =R(\fp).$

(2) If $\fp$ is maximal, $k(\fp) = A/\fp$ and 
Corollary \ref{cor.rex} (with $N=B =A$) implies
 $\hat M(\fp) = M \hat \otimes_A A/\fp  = (M \hat\otimes_A A)/ \fp (M \hat\otimes_A A) = M^\wedge/\fp M^\wedge=M^\wedge(\fp).$ 
If $R$ is complete, then $M^\wedge=M$ by Corollary \ref{cor.cpl} (iv) and hence $M^\wedge(\fp)=M(\fp)$.
\end{proof}

For $R$ Noetherian and complete and $\fp$ maximal we have equality of the completed fiber 
$\hat R(\fp)$ and the usual fiber $R(\fp)$ by Corollary \ref{cor.maxA}. For a non-maximal prime $\fp \in A$ however, 
the usual fiber and the completed fiber may be completely different (see Example \ref{ex.sem0}). 
\medskip

From the first sequence of Corollary \ref{cor.rex}  we get
\begin{corollary}\label{cor.pres}
Let  $ R^p \xrightarrow{\text {$T$}} R^q \to  M \to 0$ be an  $R$-presentation of $M$.
Then 
$$ (R \hat\otimes_A B)^p \xrightarrow{\text {$T$}} (R \hat\otimes_A B)^q \to  M \hat\otimes_A B\to 0$$
 is an  $R \hat\otimes_A B$-presentation of $M\hat\otimes_A B$.
 In particular,
$\hat M(\fp) = \coker(\hat R(\fp)^p \xrightarrow{\text {$T(\fp)$}}\hat R(\fp)^q)$.
\end{corollary}

The following corollary says that "Fitting ideals commute with completed base change".

\begin{corollary} \label{cor.fitt}
Let $A \to R$ be a ring map and
$$ R^p \xrightarrow{\text {$T$}} R^q \to M \to 0$$
a presentation of the  $R$-module $M$.  The 
ideal generated by all $q-j$ minors of $T$ is denoted by $F_j(M)$ and called the {\em j-th  Fitting ideal} of M. Then, for any ring map $A \to B$,
$$ F_j(M \hat\otimes_A B) =  F_j(M) \cdot R\hat\otimes_A B. $$
In particular, for $\fp \in \Spec A$  and $B = k(\fp)$ we get
$ F_j(\hat M(\fp)) =  F_j(M) \cdot \hat R(\fp)$.
\end{corollary}
 
\begin{proof}
This follows since $ (R \hat\otimes_A B)^p \xrightarrow{\text {$T$}} (R \hat\otimes_A B)^q \to M \hat\otimes_A B \to 0$ is a presentation of $M \hat\otimes_A B$ by Corollary \ref{cor.pres}.
\end{proof}


\section {The completed fiber as semi-local ring}\label{ssec.sem}

We assume that $A \to R$ is a morphism of Noetherian rings,  $I$ an ideal of $R$, and $M$ a finite $R$-module.
For   $\fp \in \Spec A$ let $\hat M_\fp = (M_\fp)^\wedge$ resp. $\hat M(\fp) =M(\fp) ^\wedge$ 
be the $I$-adic completed stalk, resp. completed fiber of $M$ over $\fp$ (Definition \ref{def.cf}). 

In this section we study  $\hat M_\fp$ and  $\hat M(\fp)$ under the  assumption that $R/I$  is finite over $A$. 
We recall first some general results on semi-local rings.

\begin{lemma}\label{lem.sem0}
Let $A\to S$ be a finite morphism of  Noetherian  rings and $\fp \in \Spec A$. 
\begin{enumerate}
\item The fiber $S(\fp)$ is (if non-zero) an Artinian semi-local ring with maximal ideals, say, $\bar\fm_1,..., \bar\fm_r$, and  the (finite) direct sum of its localizations at $\bar\fm_i$. 
\item The stalk $S_\fp$ is a semi-local ring with $\fm_i = \bar\fm_i \cap S_\fp$, $i=1,...,r$,   the maximal ideals of $S_\fp$. We have $\fm_i = \tilde \fm_i S_\fp$, where  $\tilde \fm_i = \fm_i \cap S$ is a prime ideal  of $S$  satisfying  $\tilde \fm_i  \cap A = \fm_i  \cap A = \fp$.\\
Moreover, the Jacobson radical is $\rad(S_\fp)=\sqrt{\fp S_\fp} = \fm_1\cap...\cap \fm_r$.
Note that $S_\fm = S$ if $(A,\fm)$ is local.
\item If $\fp$ is a {\em maximal} ideal of $A$, then the preimages $\tilde \fm_i = \bar\fm_i \cap S$ are maximal ideals of $S$\,\footnote{This is wrong if $\fp$ is not maximal: Consider $A=\mathbb Z[x]$, $R=A[[y]]$, $I=\langle y\rangle$ and $\fp=\langle p\rangle$ for some prime $p\in \mathbb Z$. The ideal $\fm=\langle p,y\rangle$ is maximal in $R_\fp$. The induced ideal in $R$ is $\langle p,y\rangle$, which is not maximal.}, and these are all maximal ideals of $S$ that lie over $\fp$.
\end{enumerate}
\end{lemma}

\begin{proof}
1. $S(\fp)$ is finite dimensional over the field $k(\fp)$ and hence Artinian (\cite[00J6]{Sta19}), it has only finitely many maximal ideals (\cite[00J7]{Sta19}), say 
$\bar\fm_1,...,\bar\fm_r$.  By  \cite[00JB]{Sta19} any prime ideal is a maximal ideal and $S(\fp)$ is semi-local and the finite direct sum of the localizations at its maximal ideals.

2. Since the $\fm_i$ are the preimages in $S_\fp$ of the 
$\bar\fm_i \in S_\fp/\fp S_\fp$,
the  prime ideals $\fm_i$ are in fact maximal ideals, since
 $\fm_i$  contains $\fp S_\fp$ and hence
$S_\fp/\fm_i = S_\fp/\fm_i +\fp S_\fp=(S_\fp/\fp S_\fp)/\bar \fm_i$ is a field. Thus $\fm_i$
is a maximal ideal.

Now let  $\fm$ be any maximal ideal of $S_\fp$. We claim $ \fm \cap A_\fp=\fp A_\fp$. In fact,
$\fm \cap A_\fp$ is a prime ideal of $A_\fp$ of the form $\fq A_\fp$ with $\fq\subset \fp$ a prime ideal in $A$ (\cite[Theorem 4.1]{Mat86}). Then 
$ \fm \cap A_\fp=\fq A_\fp \subset \fp A_\fp$. Since $A_\fp \to S_\fp$ is finite, there exists (by going-up, \cite[Theorem 9.4]{Mat86}) a prime ideal $\fm'$ of $S_\fp$, $\fm\subset \fm'$ s.t. $\fm' \cap A_\fp = \fp A_\fp$.
Since $\fm$ is maximal, $\fm' = \fm$ and the preimage of 
$\fm$ in $A_\fp$ is
$ \fm \cap A_\fp=\fp A_\fp$. Therefore, the ideal in $S_\fp$ generated by the image of  $\fp A_\fp$, which is 
$\fp S_\fp$, is contained in $\fm$.
Then $\bar \fm = \fm/\fp S_\fp$ is a maximal ideal of $S(\fp)$ and hence one of the $\bar\fm_i$ by 1.. It follows that $\fm$ is one of the $\fm_i$ and that $S_\fp$ is semi-local.

We show now that $\fm_i = \tilde \fm_i S_\fp$. Since $\fm_i$ contains $\fp S_\fp$,
the preimage $\tilde \fm_i \subset S$ of $\fm_i$ 
is a prime ideal that contains $\fp$ and hence $\tilde \fm_i S_\fp$ contains $\fp S_\fp$. By \cite[Theorem 4.1]{Mat86}
$\tilde \fm_i S_\fp$ is a prime ideal of $S_\fp$ and it follows that  $S_\fp/\tilde \fm_i S_\fp=S(\fp)/\tilde \fm_i S(\fp)$ is integral. This hows that $\tilde \fm_i S(\fp)$ is a prime ideal of $S(\fp)$ and hence a maximal ideal by 1.. By the reasoning above $\tilde \fm_i S_\fp =\fm_i$.


With $S(\fp)$ also its reduction $S_\fp/K$, 
$K=\sqrt{\fp S_\fp}$, is semi-local  with the same maximal ideals.
We have 
$K=\fm_1\cap\ldots\cap \fm_r,$
 since $K$ is the preimage in $S_\fp$ of 
$\bar\fm_1\cap\ldots\cap \bar\fm_r = \langle 0\rangle$ (the nilradical of $S_\fp/K$, which is 0). \\
\indent 
3. If $\fp$ is maximal, then $S(\fp)= S/\fp S$. Let $\tilde\fm$ be a maximal ideal of $S$ lying over $\fp$. Then $\fp = A\cap \tilde\fm$, hence $\fp S \subset \tilde\fm$ and
$\bar \fm = \tilde \fm/\fp S$ is a maximal ideal of $S(\fp)$. Therefore $\tilde\fm$ is one of the $\tilde\fm_i$ as above.
\end{proof}

\begin{lemma}\label{lem.sem}
Let $S$ be a semi-local ring with maximal ideals $\fm_1, \ldots ,\fm_r$ and let $I=\rad(S)=\fm_1\cap\ldots\cap \fm_r$ (the Jacobson radical of $S$). Then 
\begin{enumerate}
\item $I=\fm_1\cdot\ldots\cdot \fm_r$ and the $I$-adic completion $S^{\wedge,I}$ is a direct sum of complete local rings 
$(S_{\fm_i})^{\wedge,\fm_i}\cong (S^{\wedge,I})^{\wedge,\fm_i}= S^{\wedge,\fm_i}\cong (S^{\wedge,I})_{\fm_i} $,
$$S^{\wedge,I} \cong  (S_{\fm_1})^{\wedge,\fm_1}\oplus...\oplus (S_{\fm_r})^{\wedge,\fm_r}.$$
\item $S/I$ is semi-local and a direct sum of fields,
$$S/I \cong S/\fm_1\oplus...\oplus S/\fm_r \cong S_{\fm_1}/IS_{\fm_1}\oplus\ldots\oplus S_{\fm_r}/IS_{\fm_r}.$$
\item If $S$ is Henselian with respect to $I$,
then $S\cong S_{\fm_1}\oplus\ldots\oplus S_{\fm_r}$.
\end{enumerate}
\end{lemma}
\begin{proof} Statement 1. follows from \cite[Theorem 1.3 and Theorem 8.15]{Mat86}. Since $I \subset \fm_i$ we have $(S^{\wedge,I})^{\wedge,\fm_i}= S^{\wedge,\fm_i}$.
\\
2. $S/I$ is semi-local and isomorphic to $S/\fm_1\oplus...\oplus S/\fm_r$ by \cite[Theorem 1.3 and Theorem 1.4]{Mat86},
and hence an Artinian ring. This implies by \cite[Lemma 10.52.6, tag 00JB]{Sta19}, that 
$$S/I\cong S_{\fm_1}/IS_{\fm_1}\oplus\ldots\oplus S_{\fm_r}/IS_{\fm_r}.$$
3. If $S$ is Henselian with respect to $I$ this isomorphism lifts to 
$$S\cong S_{\fm_1}\oplus\ldots\oplus S_{\fm_r} $$
(\cite{Sta19} Lemma  15.11.6 (2), tag 09XI or \cite{KPR75} Satz 2.6.1 (3)).
\end{proof}

In general the localization of a complete local ring is not complete, but in the following situation it is.
\begin{lemma} \label{lem.sem2}
Let $S$ be a ring,  $\fm_1,...,\fm_r$ distinct maximal ideals of $S$, and $I=\fm_1\cap ...\cap\fm_r$. Then 
\begin{enumerate}
\item $I=\fm_1\cdot...\cdot\fm_r$ and
$S^{\wedge,I}$ is a semi-local ring with maximal ideals 
$\tilde \fm_i :=\fm_i S^{\wedge,I}$, $i=1,...,r$, Jacobson radical $\rad(S^{\wedge,I}) = I S^{\wedge,I} = \tilde \fm_1\cap ...\cap\tilde\fm_r$, 
and
$$ S^{\wedge,I} \cong  S^{\wedge,{\fm_1}} \oplus ...\oplus S^{\wedge,{\fm_r}}.$$
For each $i$, $S^{\wedge,{\fm_i}} =(S^{\wedge,I})^{\wedge,\tilde\fm_i}\cong (S^{\wedge,I})_{\tilde\fm_i} \cong  (S_{\fm_i})^{\wedge,{\fm_i}}$ is a complete local ring with maximal ideal $\fm_i S^{\wedge,\fm_i}$ and residue field $S^{\wedge,\fm_i}/ \fm_i S^{\wedge,\fm_i}= S/\fm_i$.
Moreover, $S^{\wedge,I}/ IS^{\wedge,I}= S/I = S/\fm_1 \oplus...\oplus S/\fm_r$. 

\item If $S$ is Noetherian and $M$ a finite  $S$-module, then $(M^{\wedge,I})_{\fm_i} \cong  M^{\wedge,{\fm_i}}= (M_{\fm_i})^{\wedge,{\fm_i}}$ and
$$ M^{\wedge,I} \cong  M^{\wedge,{\fm_1}} \oplus ...\oplus M^{\wedge,{\fm_r}}.$$
\end{enumerate}
\end{lemma}

\begin{proof} 1. Write $S^{\wedge} =S^{\wedge,I}$. 
Since $IS^\wedge \subset \fm_i S^\wedge$,
$$S/\fm_i = (S/I)/\fm_i =
(S^{\wedge} /I S^\wedge)/\fm_i S^\wedge = S^\wedge /\tilde\fm_i$$ 
is a field, $\tilde\fm_i$ is a maximal ideal of $S^\wedge$. 
Since the $\fm_i$ are pairwise coprime, $I=\fm_1\cdot...\cdot\fm_r$ (\cite[Theorem 1.3]{Mat86}) we have
$IS^{\wedge} = \fm_1S^{\wedge} \cdot ... \cdot  \fm_rS^{\wedge} = \tilde\fm_1\cdot...\cdot \tilde\fm_r
= \tilde\fm_1\cap...\cap \tilde\fm_r$. 
By Lemma \ref{lem.compl}.4. we have
$IS^{\wedge} \subset \rad(S^{\wedge})$. Since obviously  
$\rad(S^{\wedge})\subset IS^{\wedge}$, we get equality.
This shows that $S^{\wedge}$ is semi-local with maximal ideals $\tilde\fm_1,...,\tilde\fm_r$. Moreover, $S^\wedge = (S^\wedge)^{\wedge,IS^\wedge}$ and 
$S^{\wedge,\fm_i} = (S^{\wedge, I})^{\wedge,\tilde\fm_i}$
(Lemma \ref{lem.compl}.2. and \ref{lem.compl}.11.). Then  Lemma \ref{lem.sem}
gives 
$ S^{\wedge,I} \cong  S^{\wedge,{\fm_1}} \oplus ...\oplus S^{\wedge,{\fm_r}}$, the isomorphisms for 
$S^{\wedge,{\fm_i}}$ and the direct sum for 
$S^{\wedge,I}/ I S^{\wedge,I}$.

2. By Lemma \ref{lem.compl}.6. we have  
$M^{\wedge, I}\cong M \otimes_S S^{\wedge,I}$ and
$M\otimes_S S^{\wedge, \fm_i} \cong M^{\wedge, \fm_i}$. The result for $M^{\wedge,I}$ follows then from the result for $S^{\wedge,I}$.
\end{proof}

\begin{proposition}\label{prop.sem}
Let  $A \to R$ be a morphism of Noetherian rings,  $I\subset R$ an ideal such that $R/I$ is finite over $A$, and $\fp$ a prime ideal of $A$.  
Then:
\medskip

\begin{enumerate}
\item $R_\fp/\fp  R_\fp+I R_\fp= \hat R_\fp/\fp \hat R_\fp+I \hat R_\fp$, $\hat R_\fp = (R_\fp)^{\wedge,I}$, is (if non-zero) an Artinian semi-local ring, with maximal ideals $\bar \fm_1,...,\bar \fm_r$ and isomorphic to the direct sum of its localizations at $\bar \fm_i$.\\

\item The ideals  $\fn_i := \bar \fm_i \cap R_\fp$, resp. $\hat \fn_i := \bar \fm_i \cap \hat R_\fp$, $i=1,...,r$, are maximal ideals of $R_\fp$ resp. $\hat R_\fp$.
The ideals $K:=\sqrt{\fp R_\fp + I R_\fp} = \fn_1\cap...\cap \fn_r$ resp. $\hat K := \sqrt{\fp \hat R_\fp + I \hat R_\fp} = \hat \fn_1\cap\ldots\cap \hat \fn_r$ satisfy
$$
\begin{array}{llll}
R_\fp / K &=R_\fp/\fn_1 \oplus...\oplus R_\fp/\fn_r,\\
\hat R_\fp / \hat K &= \hat R_\fp/\hat\fn_1 \oplus...\oplus \hat R_\fp/\hat\fn_r.
\end{array}
$$
\medskip
\item Let $\tilde R_\fp := (R_\fp)^{\wedge, \fp}$ be the $\fp R_\fp$-adic completion of $R_\fp$\footnote{$\tilde R_\fp$  may be considered as a formal neighbourhood of the usual fiber $R_\fp$.}.
The $I$-adic completion $(\tilde R_\fp)^{\wedge,I}$ 
of $\tilde R_\fp$ is (if non-zero) a semi-local ring with maximal ideals 
$\tilde\fm_1,\ldots,\tilde\fm_r$,  $\tilde\fm_i = (\tilde R_\fp)^{\wedge,I} \cap \bar \fm_i = \fn_i (\tilde R_\fp)^{\wedge,I}$,
and  Jacobson radical
$\rad((\tilde R_\fp)^{\wedge,I}) =
\tilde \fm_1\cap\ldots\cap \tilde \fm_r = 
\sqrt{\fp (\tilde R_\fp)^{\wedge,I}  + I (\tilde R_\fp)^{\wedge,I}}$. 
We have a direct sum 
$$(R_\fp)^{\wedge, \fp + I}=(\tilde R_\fp)^{\wedge,I} \cong  (R_\fp)^{\wedge,\fn_1} \oplus ...\oplus (R_\fp)^{\wedge, \fn_r},$$
with $(R_\fp)^{\wedge,{ \fn_i}} \cong ((R_\fp)_{\fn_i})^{\wedge,{ \fn_i}}
\cong ((\tilde R_\fp)_{\tilde \fm_i})^{\wedge,{\tilde \fm_i}}\cong  (\tilde R_\fp)^{\wedge,{\tilde \fm_i}} \cong ((\tilde R_\fp)^{\wedge,I})_{\tilde \fm_i}$ being  complete local rings. Moreover $(\tilde R_\fp/\fp \tilde R_\fp)^{\wedge,I}=(R_\fp/\fp R_\fp)^{\wedge,I}= \hat R(\fp)$. 
\\
If $M$ is a finite $R$-module, then  
$$ (M_\fp)^{\wedge,\fp+I} \cong  (M_\fp)^{\wedge,\fn_1} \oplus ...\oplus (M_\fp)^{\wedge, \fn_r},$$
with $(M_\fp)^{\wedge,{ \fn_i}}\cong M_\fp \otimes_{R_\fp}(R_\fp)^{\wedge,{\fn_i}} \cong ((M_\fp)_{\fn_i})^{\wedge,{ \fn_i}} =(M_{\fn_i})^{\wedge,{ \fn_i}} = M^{\wedge,{ \fn_i}} $.\\

\item The completed fiber $\hat R(\fp)=R(\fp)^{\wedge,I}$ is (if non-zero) a semi-local ring with maximal ideals $\fm_i = \hat R(\fp) \cap\bar\fm_i = \fn_i \hat R(\fp)$, $i=1,...,r$, and  $\rad(\hat R(\fp)) =\fm_1\cap\ldots\cap \fm_r = \sqrt{I \hat R(\fp)}$.  
The localizations $\hat R(\fp)_{\fm_i}
\cong  (\hat R(\fp)_{\fm_i})^{\wedge,\fm_i}\cong \hat R(\fp)^{\wedge,\fm_i}\cong (R(\fp)_{\fn_i})^{\wedge,\fn_i} \cong R(\fp)^{\wedge,\fn_i}$ are complete local rings
and $\hat R(\fp)$ is a direct sum
$$\hat R(\fp) \cong  R(\fp)^{\wedge,\fn_1}\oplus\ldots\oplus R(\fp)^{\wedge,\fn_r}.$$
The residue field $k(\fm_i):= \hat R(\fp)_{\fm_i}/\fm_i \hat R(\fp)_{\fm_i}\cong R_\fp/\fn_i R_\fp$ of $\hat R(\fp)_{\fm_i}$
is  finite over $k(\fp)$.\\
If $M$ is a finite $R$-module, then  
$$\hat M(\fp)\cong \hat M(\fp)_{\fm_1}\oplus\ldots\oplus \hat M(\fp)_{\fm_r}$$
and $\hat M(\fp)_{\fm_i} \cong  (\hat M(\fp)_{\fm_i})^{\wedge,{\fm_i}} \cong \hat M(\fp)^{\wedge,{\fm_i}} \cong M(\fp)^{\wedge,{\fn_i}} = (M(\fp)_{\fn_i})^{\wedge,{\fn_i}} $.  
Moreover
$ \hat M(\fp)= (M_\fp/\fp M_\fp)^{\wedge,I}= (M_\fp/\fp M_\fp)^{\wedge,\fp+I}$. 
\end{enumerate}
\end{proposition}

\begin{proof}
The following commutative diagram makes the relation between the different rings and maximal ideals clear. 
\[
\begin{xymatrix}{
R_\fp \ar[r]^{\fn_i \quad\quad \hat \fn_i} \ar@{->>}[d]& \hat R_\fp \ar[r]^{ \quad\quad \tilde \fm_i} \ar@{->>}[d]_{\hat\fn_i \mapsto \fm_i  }
& (\tilde R_\fp)^{\wedge, I} \ar@{->>}[d]^{\tilde \fm_i \mapsto \bar \fm_i} \ar@{->>}[dl] \\
R(\fp) \ar[r] & \hat R(\fp) \ar@{->>}[r] &  {\hat R(\fp)/I \hat R(\fp)} & \mkern -46mu = \hat R_\fp/\fp \hat R_\fp+I \hat R_\fp
}
\end{xymatrix}
\]
The maximal ideals $\bar \fm_i \subset \hat R(\fp)/I \hat R(\fp)
= R(\fp)/I R(\fp)$ are the starting point, the other maximal ideals $\fn_i \subset  R_\fp$, $\hat \fn_i \subset  \hat R_\fp$, $\tilde \fm_i \subset  (\tilde R_\fp)^{\wedge,I}$, and $\fm_i \subset  \hat R(\fp)$ are the preimages of  $\bar \fm_i$ (that these are maximal ideals will be shown below). The last upper horizontal arrow is $\hat R_\fp = (R_\fp)^{\wedge, I} \to (R_\fp)^{\wedge, I+\fp} = (\tilde R_\fp)^{\wedge,I}$ and the
diagonal arrow is the composition 
$(\tilde R_\fp)^{\wedge, I} \twoheadrightarrow (\tilde R_\fp/\fp \tilde R_\fp)^{\wedge,I}=
((R_\fp)^{\wedge,\fp} /\fp (R_\fp)^{\wedge,\fp} )^{\wedge,I}
=
(R_\fp/\fp R_\fp)^{\wedge,I}= \hat R(\fp)$ 
(using Lemma \ref{lem.compl}.3.). The other arrows are the canonical ones.
\\

1.  Since $S:=R/I$ is finite over $A$, the fiber $S(\fp)= R_\fp/\fp  R_\fp+I R_\fp
=\hat R_\fp/\fp  \hat R_\fp+I \hat R_\fp$ is a semi-local direct sum with maximal ideals $\bar\fm_1,...,\bar\fm_r$ by Lemma \ref{lem.sem0}.1..

2. $S_\fp = R_\fp/I R_\fp=\hat R_\fp/I \hat R_\fp$ is a semi-local direct sum with maximal ideals $S_\fp \cap \bar \fm_i$, $i=1,...,r$. As a preimage from $S_\fp$, $\fn_i = R_\fp \cap \bar \fm_i$ contains $\fp R_\fp+IR_\fp$ and hence $R_\fp /\fn_i= S(\fp)/ \bar \fm_i$ is a field, showing that $\fn_i$ is a maximal ideal. The same holds for the ideals $\hat\fn_i$.
$R_\fp/\fp  R_\fp+I R_\fp$ is an Artinian semi-local ring and so is its reduction $R_\fp/K$, which is therefore the direct sum of the fields $R_\fp/\fn_i$. 
Hence $K=\sqrt{\fp  R_\fp+I R_\fp}=\fn_1\cap...\cap\fn_r$.

3. By Lemma \ref{lem.compl}.11.  and  \ref{lem.compl}.1. we have $(\tilde R_\fp)^{\wedge,I} = (R_\fp) ^{\wedge, \fp + I} = (R_\fp)^{\wedge,K}$. $K=\sqrt{\fp R_\fp + I R_\fp}$. 
We deduce from Lemma \ref{lem.sem2} that 
$((R_\fp)^{\wedge,K})_{\tilde\fm_i}\cong (R_\fp)^{\wedge,{\fn_i}} 
\cong ((R_\fp)_{\fn_i})^{\wedge,{\fn_i}}$ 
is semi-local with maximal ideals 
$\fn_i(R_\fp)^{\wedge,K}=\fn_i(\tilde R_\fp)^{\wedge,I}$. Since $\fn_i = \tilde\fm_i \cap R_\fp$, the maximal ideal $\fn_i(\tilde R_\fp)^{\wedge,I}$ is contained 
in $\tilde\fm_i$ and hence  
$\tilde \fm_i =\fn_i(\tilde R_\fp)^{\wedge,I}$.
Moreover,  Lemma \ref{lem.sem2} gives
 $(\tilde R_\fp)^{\wedge,\tilde\fm_i} = (R_\fp)^{\wedge,\fn_i}
= ((R_\fp)_{\fn_i})^{\wedge,{\fn_i}}$, 
$$(\tilde R_\fp)^{\wedge,I} =\oplus_{i=1}^r (\tilde R_\fp)^{\wedge,\tilde\fm_i}=\oplus_{i=1}^r (R_\fp)^{\wedge,\fn_i},$$
and $\tilde \fm_1\cap\ldots\cap \tilde \fm_r = K (\tilde R_\fp)^{\wedge,I}$= 
$\sqrt{\fp (\tilde R_\fp)^{\wedge,I} + I (\tilde R_\fp)^{\wedge,I}}$. 

4. Since $\tilde R_\fp/\fp \tilde R_\fp= R_\fp/ \fp R_\fp = R(\fp)$ (Lemma \ref{lem.compl}.3.) and since $(\tilde R_\fp)^{\wedge,I}$
is semi-local with maximal ideals $\tilde\fm_i$
by 3., we get that $(\tilde R_\fp)^{\wedge,I}/\fp (\tilde R_\fp)^{\wedge,I}=
R(\fp)^{\wedge,I}= \hat R(\fp)$ is semi-local 
with maximal ideals $\tilde \fm_i \hat R(\fp)$
$i=1,...,r$. Since $\fm_i = \hat R(\fp) \cap \bar \fm_i 
= \hat R(\fp) \cap \tilde \fm_i $ (by the diagram above) we get $\tilde \fm_i \hat R(\fp)=\fm_i$.
The remaining statements follow from step 3.
For the residue field $k(\fm_i)$ we get $\hat R(\fp)_{\fm_i}/\fm_i \hat R(\fp)_{\fm_i}\cong R(\fp)^{\wedge,\fn_i}/\fn_i R(\fp)^{\wedge,\fn_i} \cong (R(\fp)/\fn_i R(\fp))^{\wedge,\fn_i} \cong (R_\fp/\fn_i R_\fp))^{\wedge,\fn_i} \cong 
 R_\fp/\fn_i R_\fp.$
 It is finite over $k(\fp)$ since $\hat R(\fp)/I\hat R(\fp)=R(\fp)/IR(\fp)$ is finite over $k(\fp)$ and $I\hat R(\fp) \subset \fm_i$.

Since $\hat M(\fp)= M(\fp)^{\wedge,I} = M(\fp)\otimes_{R(\fp)}R(\fp)^{\wedge,I}= M(\fp)\otimes_{R(\fp)} \hat R(\fp)$
(Lemma \ref{lem.compl}.6. and Corollary \ref{cor.cpl}) and 
$\hat M(\fp)_{\fm_i}=
\hat M(\fp)\otimes_{\hat R(\fp)}\hat R(\fp)_{\fm_i} $
 (\cite[Lemma 10.11.15, tag 00DK]{Sta19})
we get
$$\hat M(\fp) \cong \hat M(\fp)_{\fm_1}\oplus\ldots\oplus \hat M(\fp)_{\fm_r}$$
and (by Lemma \ref{lem.sem2} and since $I\hat R(\fp) \subset \fm_i$) $\hat M(\fp)_{\fm_i} \cong  (\hat M(\fp)_{\fm_i})^{\wedge,{\fm_i}} \cong \hat M(\fp)^{\wedge,{\fm_i}} \cong M(\fp)^{\wedge,{\fn_i}} $.
\end{proof}

\begin{remark}\label{rm.sem0}{\em
(1) It is shown in Lemma \ref{lem.open} that $\hat M(\fp) \ne 0$ iff $\fp \in \Supp_A(M/I)$ and that 
$$\Supp_R(M/IM) = \Supp_R(M)\cap V(I) \to 
\Supp_A(M/I) =\{\fp \in \Spec A \mid \hat M(\fp) \ne 0 \},$$
$\fn \mapsto \fp = \fn\cap A$, is surjective and finite. \\
It follows from Proposition \ref{prop.sem} for $M=R$ that the prime ideals in $V(I)\subset \Spec R$ lying over $\fp \in \Supp_A(R/I)$ are exactly the $\tilde \fn_i =R\cap \bar\fm_i =R\cap \fm_i$, $i=1,...,r_\fp$, with $\bar\fm_i$ resp. $\fm_i$ the maximal ideals of $\hat R(\fp)/I$ resp. of  $\hat R(\fp)$ from Proposition \ref{prop.sem}.1. 
It follows for $\fn \in \Spec R$, using 
$\hat M(\fp)=M(\fp)\otimes_{R(\fp)}\hat R(\fp)$,
that 
$$\hat M(\fp)_\fn \ne 0 \iff \fn \in \Supp_R(M)\cap V(I),
\ \fp = \fn \cap A.$$
The maximal ideals $\fn_i =R_\fp \cap \bar \fm_i$, $i=1,...,{r_\fp}$, of $R_\fp$ are exactly the 
$\tilde \fn_i R_\fp$. The ideal $\tilde \fn_i$ is a prime ideal in $R$ containing $\fp + I$. It is a maximal ideal if $\fp$ is a maximal ideal of $A$ by Lemma \ref{lem.sem0}, since $R/\tilde \fn_i = R/I+\tilde \fn_i$ and $R/I$ is finite over $A$ (see also the footnote there).

(2) While  $\hat R_\fp$ and $\hat R(\fp)$  are semi-local, this may not  be the case with $R_\fp$ and $R(\fp)$ for non-maximal $\fp$. See Example \ref{ex.sem0} (1), where $R_\fp=R(\fp)$ has infinitely many maximal ideals.
}
\end{remark}

\begin{example}\label{ex.sem0}
{\em The following examples illustrate Proposition \ref{prop.sem} and show the difference between usual fiber and the completed fiber for a non-maximal prime $\fp \subset A$. 
\begin{enumerate}
\item [(1)] In contrast to the completed fiber $\hat R(\fp)$, the usual fibre $R(\fp)$ may not be semi-local:
Consider $A=K[t]$,  $K$ a field, $I=\x \subset A[x]$, $R=A[[x]]$ with $t$ and $x$ one variable. For
$\fp =\langle 0 \rangle \in  \Spec A$, we have $A_\fp = k(\fp) = K(t)$, and
$$R(\fp) = R_\fp  = K[t][[x]] \otimes_{K[t]} K(t) =  \{g / h \, | \, g \in K[t][[x]],  h \in K[t]\smallsetminus 0\},$$
a subring strictly contained in  $\hat R(\fp)=\hat R_\fp = R\hat\otimes_A A_\fp  = K(t)[[x]]$ (see also \cite[Example 2.2]{GP21}). For $c\in K$,
$K[t][[x]]/\langle t-c-x\rangle\cong K[[t]]$ via $f(t,x)\mapsto f(t+c,t)$ and hence $R_{\fp}/\langle t-c-x\rangle= (K[t][[x]]/\langle t-c-x\rangle)_\fp 
\cong K[[t]]_{\langle 0 \rangle}  \cong K((t))$
is a field. Hence $R(\fp) = R_\fp$ has infinitely many maximal 
 ideals $\langle t-c-x\rangle$ if $K$ is infinite. On the other hand, $\hat R(\fp)$ is a local ring.

\item [(2)] A module $M$ with $\hat M_\fp =\hat M(\fp) =0$ while $\dim_{k(\fp)} M(\fp) = \infty$ (see \cite[Example 2.4]{GP21} for more details):\\
Let $A=\Z$, $I=\x \subset \Z[x]$, $R = \Z[x]^{\wedge,I}=\Z[[x]]$, and $M = R/\langle x-p\rangle$, $p \in \Z$ a prime number. Since $R= \lim\limits_{\longleftarrow}\Z[x]/\langle x \rangle^n $ we obtain
$M = \lim\limits_{\longleftarrow}\Z/p^n = \hat \Z_{\langle p\rangle}$, the ring of 
$p$-adic integers.
For the prime ideal $\langle 0\rangle$ we get 
$\hat M(\langle 0\rangle)= M\hat \otimes_\Z \Q =\Q[[x]]/\langle x-p\rangle=0,$
while
$M(\langle 0\rangle)= M\otimes_\Z \Q = \hat \Z_{\langle p\rangle}\otimes_\Z \Q =\Quot(\hat \Z_{\langle p\rangle}) $ 
is infinite dimensional over $\Q=k(\langle 0\rangle)$. 

\item [(3)] Let  $A=K[t]$,  $K$ an algebraically closed field, $I=\x \subset A[x]$, $R=A[x]^{\wedge,\x}=A[[x]]$ with $t$ and $x$ one variable, and 
$M= R/\langle t-x \rangle \cong K[[t]] $ as $A$-module. We have $\Supp_A (M) = \Spec A$ and
$$ \{\fp \,|\,\hat d_\fp(M)\neq0\} = \{\langle t \rangle\} \subsetneqq
\{\fp \,|\,d_\fp(M)\neq0\}= \{\langle t \rangle , \langle 0 \rangle\}  \subsetneqq \Supp_A (M) $$ (see \cite[Example 2.3 (4)]{GP21}).
\end{enumerate}
 }
\end{example}

Proposition \ref{prop.sem} says that for $R/I$ finite over $A$ the completed fiber $\hat R(\fp)$ of $A\to R$ over $\fp \in \Spec A$ is semi-local with maximal ideal $\fm_1,...,\fm_r$, and decomposes as a direct sum of complete local  rings $\hat R(\fp)^{\wedge,\fm_i}$. The following proposition improves this 
in so far as we can find,  for $\fp$ a {\em maximal ideal}, open neighbourhoods $\Spec A_i$ of $\fp$ and open sets
 $\Spec Q_i \subset \Spec R$, separating the maximal ideals $\fm_i$,  
such that $Q_i/IQ_i$ is finite over $A_i$ and satisfies again the assumptions (and hence the conclusions) of Proposition  \ref{prop.sem}.

\begin{proposition}\label{prop.sem2}
In addition to  the assumptions of Proposition \ref{prop.sem}
 assume moreover that 
$\fp\in \Spec A$  is a {\em maximal} ideal.  Let  $\fm_1,...,\fm_r$  be the maximal ideals of the semi-local ring 
$\hat R(\fp) = \hat R(\fp)_{\fm_1}\oplus\ldots\oplus \hat R(\fp)_{\fm_r}$
(Proposition \ref{prop.sem}). Denote by $\varphi:A\to R$ the given map.
\begin{enumerate}
\item 
For each $1\le i \le r $ the preimage $\tilde\fn_i$ in $R$ of $\fm_i$ is a maximal ideal, and there exists an  open neigbourhood  $\Spec A_i \subset \Spec A$ of $\fp$, 
$A_i = A_{g_i}$ for some $g_i \in A$,
and 
 an open neigbourhood $U_i=\Spec Q_i \subset \Spec R$ of $\tilde\fn_i$,  $Q_i = R_{\varphi(g_i)}$,
 such that  
 $\tilde\fn_j \notin U_i$ if $j\ne i$, 
and  $Q_i/IQ_i$ is finite over $A_i$. \\
Moreover, for $\fq \in \Spec A_i$ such that $\hat R(\fq) \ne 0$, the ring $\hat Q_i(\fq)$ is semi-local with maximal ideals being a subset of the maximal ideals of the semi-local ring $\hat R(\fq)$.
\item Assume furthermore that $R$ is  of finite type and flat over $A$.
Then $Q_i$ and $A_i$ can be chosen such that 
$\Spec Q_i \to \Spec A_i$ is flat and open with non-empty fibers $\Spec Q_i(\fq)$.
\end{enumerate}
\end{proposition}

\begin{proof}
1. Consider the following sequence,
$$R\longrightarrow R_\fp\longrightarrow (\tilde R_\fp)^{\wedge,I}
\longrightarrow  \hat R(\fp) = (\tilde R_\fp)^{\wedge,I}/\fp (\tilde R_\fp)^{\wedge,I}.$$
Denote by $\tilde\fn_i$ the preimage in $R$ of $\fm_i = \tilde \fm_i \hat R(\fp)$, which is also the preimage of $\tilde \fm_i \subset (\tilde R_\fp)^{\wedge,I}$. 
Then  $\tilde\fn_i$ is a maximal ideal  by Remark \ref{rm.sem0} (1), since $\fp$ is maximal.

Now choose $f_i\in \tilde\fn_1\cap\tilde\fn_2\ldots\cap \tilde\fn_{i-1}\cap\tilde\fn_{i+1}\ldots\tilde\fn_r$ such that $f_i\notin \tilde\fn_i$  and $\tilde\fn_i \in D(f_j)$ iff $i=j$. Let $\varphi$ denote the given map $A\to R$.
$\Spec (\varphi)(D(f_i))$ can be covered by suitable affine open sets $D(g_{ik})\subset \Spec A$ and with
 $Q_{ik}:=R_{\varphi(g_{ik})}$ we get that $Q_{ik}/IQ_{ik} = (R/IR)_{\varphi(g_{ik})}$ is finite over $A_{g_{ik}}$. For a given $i$ there exists $k$ such that $\tilde\fn_i\in \Spec Q_{ik}$ and $\tilde\fn_j\notin \Spec Q_{ik}$ for $i\neq j$. With $Q_i:=Q_{ik}$ and $A_i:=A_{g_{ik}}$ 
 we can apply Proposition \ref{prop.sem} to $A_i \to Q_i$, which proves the first part of 
 the proposition.
 
2. Since $R$ is flat over $A$ and of finite type the corresponding morphism of affine schemes is open 
(\cite[00I1.]{Sta19}) and thus $Q_i$ and $A_i$ can be chosen such that 
$\Spec Q_i \to \Spec A_i$ is open with non-empty fibers.
\end{proof}

\begin{example}  \label{ex.max}{\em 
For $\fq \in \Spec A$ which is not a maximal ideal, the completed fiber $\Spec \hat R(\fq)$ may be empty while the fiber  $\Spec R(\fq)$ is non-empty (under the assumption of Proposition \ref{prop.sem2}.2.):

Let $A=K[t]$, $P= A[x]/\langle t-x \rangle =K[t] =A$, $I=\langle t \rangle$. Then
$R=P^{\wedge, I} =K[[t]]$ and with $\fq = \langle 0\rangle$
we have $R_\fq=R(\fq)=K((t))$ but, since $t$ is a unit in $K((t))$,
$\hat R_\fq=\hat R(\fq)=K((t))^{\wedge, \langle t \rangle} =\lim\limits_{\longleftarrow}K((t))/t^n=0$,  i.e. $\Spec \hat R(\fq) = \emptyset$.}
\end{example}

\begin{example}\label{ex.sem}
{\em
Consider the ideal  $I=\langle x^2+1\rangle\subset \mathbb Z[x]$ and let $R=\mathbb Z[x]^{\wedge,I}$. Let $p\in\mathbb Z$ be a prime, $\fp=\langle p \rangle \subset \mathbb Z[x]$ the corresponding prime ideal.
We determine $\hat R(\fp)$ for different choices of $\fp$.
\begin{enumerate}
\item If $p\not\equiv 1 \text{ mod } 4$ then\footnote{The polynomial $x^2+1$ splits in $ (\Z/p)[x]$ iff $p\equiv 1 \text{ mod } 4$.}
$\hat R(\fp)=(\Z/p)[x]^{\wedge,I}$ is a complete discrete valuation ring. $I$ is a maximal ideal in $(\Z/p)[x]$ and 
$(\Z/p)[x]^{\wedge,I} \cong (\Z/p)(i)[[x]]$.
\item If $p\equiv 1 \text{ mod } 4$ then                                                                                                 choose $i\in \mathbb Z/p$ with $i^2+1=0$.
$\hat R(\fp)=\mathbb (Z/p)[x]^{\wedge,I}=\mathbb (Z/p)[x]^{\wedge,{\langle x+i\rangle}} \oplus \mathbb (Z/p)[x]^{\wedge,{\langle x-i\rangle}}$ is a semi-local ring, isomorphic to $(\Z/p)[[x]] \oplus (\Z/p)[[x]]$.
\item $\hat R({\langle 0\rangle})=\mathbb Q[x]^{\wedge,I}$ is a complete discrete valuation ring, which is isomorphic to $\mathbb Q(\sqrt{-1})[[x]]$.
\end{enumerate}}
\end{example}
Note that $I\mathbb Q[x]^{\wedge,I}=\{\sum_{j\ge 0} a_j(x^2+1)^j+x\sum_{j\ge 0} b_j(x^2+1)^j, a_j,b_j\in \mathbb Q\}$.


\section{Semicontinuity of the completed fiber dimension}\label{sec.semcont}
In this section we prove the semicontinuity of
$\hat d_\fp(M)$. 
Let $A$ be a ring. We say that  a function $d: \Spec A \to \R \cup \{\infty\}$,  $\fp \mapsto d_\fp$,
 is {\em (upper) semicontinuous} at $\fp$ if $\fp$ has an open neighbourhood $U \subset \Spec A$  such that  $d_\fq  \leq  d_\fp $ for all $\fq \in U$. If this holds for all $\fp\in \Spec A$ we say that $d$ is (upper) semicontinuous on $\Spec A$. We have
$$\begin{array}{lll}
d\text{ is  semicontinuous on } \Spec A \iff  \\
\{\fp \in \Spec A| \ d_\fp \leq N\} \text{ is open in } \Spec A  \text{ for } N=0,1,2,..., \infty.
\end{array}$$
\medskip

The following semicontinuity statement should be well known.

\begin{lemma}\label{lem.free}
Let $A$ be a ring and $M$ a finitely presented  $A$-module. Then $d_\fp(M) = \dim_{k(\fp)} M(\fp)$
 is semicontinuous on  $\Spec A$. If, moreover, $M$ is flat over $A$, then $M$ is locally free and $d_\fp(M)$ is locally constant on  $\Spec A$. Conversely, if $A$ is reduced and $d_\fp(M)$ locally constant on  $\Spec A$, then $M$ is locally free.
\end{lemma}

\begin{proof} Consider a presentation of $M$,
$$ A^p \xrightarrow{\text {$T$}} A^q \to  M \to 0,$$
with matrix $T = (t_{ij}), \ t_{ij}  \in A$. 
For $\fp\in \Spec A$ 
let $T(\fp)$ be the matrix with entries the images of $t_{ij}$ in $k(\fp)$. 
Then  $\dim_{k(\fp)} M\otimes_A k(\fp) = q - \rank(T(\fp))$, which is upper semicontinuous since the rank is lower semicontinuous: If $r = \rank(T(\fp))$, then $\fp$ is contained in the open set $D(f)$, with $f \in A[t_{i,j}]$ an $r$-minor of $T$ and $f(\fp) \ne 0$ in $k(\fp)$ (and bigger minors vanish at $\fp$) . Hence, for $\fq \in D(f)$ we have $f(\fq) \ne 0$ in $k(\fq)$ and $\rank(T(\fq))\ge r$. 
If $M$ flat, it is locally free and the rank is locally constant (see \cite[Theorem 4.10 and its proof]{Mat86}).

To see the converse, we may assume that $T\neq0$ and that 
$d_\fp(M)$ is constant on $\Spec A$.
Let $I$ be the ideal generated by the non-zero $s$-minors $f_1,...,f_k$ of $T$
with $ s= \text{ max } \{l \,| \, \exists \ l\text{-minor} \neq 0 \text{ of } T\}.$
If $I\subset \fp$ for some $\fp \in \Spec A$, then $f_i \in \fp A_\fp$ and
$$d_\fp(M) > q-s.$$
Since $I\neq 0$ and $A$ is reduced, $V(I)\neq \Spec A$ and there is a $\fq \in \Spec A$ with $I\not\subset \fq$. This implies 
$$d_\fq(M) = q-s,$$
and since this dimension is constant, $V(I) =\emptyset$. Hence $I=A$ and 
the open sets $D(f_i)$ cover $\Spec A$. 

We claim that 
 $M_{f_i}$ is a free $A_{f_i}$-module on $D(f_i)$. 
 W.l.o.g. let $f_1$ be  the determinant of the upper left $s\times s$-submatrix of $T$. By a change of bases  in $A^p$ and $A^q$ over $A_{f_1}$ we can achieve that $T$ becomes
 $$T_{f_1}=\left[
\begin{array}{c|c}
1_s  & 0 \\ \hline
0 & T'
\end{array}
\right]$$
with $1_s$ the $s\times s$ unit matrix and $T'=0$ because of the maximality of $s$. Then 
$$M_{f_1} \cong A_{f_1}^q /Im(T_{f_1}) =A^{q-s}_{f_1},$$
and similar for all $f_i$. This shows that $M$ is locally free. 
 \end{proof}

\begin{example}\label{ex.notsemcont}
{\em
In many interesting applications, however, $M$ is not finitely presented over $A$. A typical case arises when $A$ is a Noetherian ring,  $R=A[[x]]=A[x]^{\wedge,\x} $, and $M$ is a finitely presented $R$-module that is not finite over $A$ (e.g. $R$ itself).

Then the semicontinuity of $d_\fp(M)$ may not hold on $\Spec A$, as the following example shows (see also Example \ref{ex.sem0} (3)).
Let $K$ be an algebraically closed field,  $ A = K[t]$, $R=A[[x]]$,  and $M  = A[[x]]/\langle x-t\rangle \cong K[[t]]$. Then the set of closed points is $\Max A = \{\langle t-c\rangle, c \in K\}$ and $\Spec A = \Max A \cup \{\langle 0 \rangle\}$.
We have $M(\fp) = K$ for $\fp= \langle t\rangle$ and $0$ for 
$\fp= \langle t-c\rangle, \ c \neq 0$, while $M(\langle 0\rangle) \cong K((t))$ and  $k(\langle 0\rangle) =K(t)$. Hence
$\dim_{k(\langle 0\rangle)}M(\langle 0\rangle) = \infty > \dim_{k(\fp)}M(\fp) $
for $\fp \neq \langle 0\rangle$ and $d_\fp(M)$ is not semicontinuous at every $\fp \ne \langle 0\rangle$  (the generic point $\langle 0\rangle$ is in the neighbourhood of any other point $\fp$). Since $\dim_{k(\langle 0\rangle)}\hat M(\langle 0\rangle) = 0$ the completed fiber dimension  $\hat d_\fp(M) = \dim_{k(\fp)} \hat M(\fp)$ is semicontinuous on $\Spec A$.
}
\end{example}
\medskip

Recall that for an arbitrary $R$-module $M$ $$\Supp_R(M) := \{\fn \in \Spec R\ | \ M_\fn \neq 0\}$$  denotes the support of $M$  and 
$\Ann_R(M) = \{f \in R \, | \, fM=0 \}$ the annihilator ideal of $M$. We have always $\Supp_R(M) \subset V(\Ann_R(M))$,
with equality if $M$ a finitely generated over $R$ (see the proof of Lemma \ref{lem.ann}). More generally, if $I\subset R$ is an ideal and if $M$ is finite over $R$ then 
 (\cite[Exercise 19. , p. 46]{AM69})
$$\Supp_R(M/IM) = V(I+\Ann_R(M)) = V(I)\cap V(\Ann_R(M)).$$
\medskip

Before we state the semicontinuity of $\hat d_\fp(M) = \dim_{k(\fp)}\hat M(\fp)$ (see Theorem \ref{thm.scont}) we compare $\hat d_\fp(M)$ with $d_\fp(M)$ in special cases and formulate a semicontinuity result, which is already non-trivial. 
The  folllowing proposition
is proved for $I=\x$ in \cite[Proposition 2.2]{GP21} 
and the proof carries over to our more general situation.

\begin{proposition}\label{prop.Afin}
Let $A \to R$ a map of Noetherian rings and $M$ a finitely generated $R$-module. If  $I\subset R$ is an ideal such that $R/I$ is finite over $A$  the following holds.
\begin{enumerate}
\item If \, $\Supp_R (M) \subset V(I)$ then  $M$ is finitely generated  over $A$ and
$\hat M(\fq) = M(\fq)$ for all $\fq \in \Spec A$. In particular,
$\hat d_\fp(M)= d_\fp(M)$ is semicontinuous on $\Spec A$.
 \item
If $M$ is finitely generated as an $A$-module,
 then  $\hat d_\fp(M)\leq d_\fp(M)<\infty$ and  $\hat d_\fp(M)$ (as well as $d_\fp(M)$) is semicontinuous on $\Spec A$.
 \end{enumerate}
\end{proposition}

 Note that  $\hat d_\fp(M) = d_\fp(M)$ for $\fp$ a maximal ideal and $R$ complete (Corollary \ref{cor.maxA} (2)). 
 The following example shows that 
$\hat d_\fp(M) < d_\fp(M)$ may happen under the assumptions of  2.  for $R$ complete and $\fp$ not maximal.  If $M$ not finite over $A$ then $\hat d_\fp(M) =0$ and $d_\fp(M) =\infty$  can happen, see Example \ref{ex.notsemcont}.

\begin{example} \label{ex.complete} {\em
Let $A=K[[t]]$, $K$ a field, $R=A[[x]]$, and let
$M = R/\langle t-x\rangle \cong K[[t]]$.  For the  two prime ideals $\langle 0\rangle$
and $\langle t\rangle$ of $A$ we get:
$k(\langle 0\rangle) = K((t))$,  $k(\langle t\rangle) = K$,
$M(\langle 0\rangle) \cong K[[t]]\otimes_{K[[t]]}K((t)) = K((t))$, $M(\langle t\rangle) \cong K$
and $d_{\langle 0\rangle}(M) =d_{\langle t\rangle}(M) =1$. Hence  $d_{\fp}(M)$
is semicontinuous (even continuous) on $\Spec A$.

$\hat M(\langle 0\rangle) \cong K((t))[[x]]/\langle t-x\rangle =0$, $\hat M(\langle t\rangle) \cong K$
and $\hat d_{\langle 0\rangle}(M) =0$, $\hat d_{\langle t\rangle}(M) =1$. 
Hence  $\hat d_{\fp}(M)$
is semicontinuous on $\Spec A$,
$M$ is finitely generated as $A$-module, and we have 
$\hat d_{\langle 0\rangle}(M) < d_{\langle 0\rangle}(M)$.
}
\end{example}

As a first step to general semicontinuity we show below (Lemma \ref{lem.open}) that the vanishing locus of
$\hat d_\fp(M)$ is open. To prove this, we need the following lemma.

\begin{lemma} \label{lem.ann}  
Let $A\to R$ be a morphism of rings and $M$ an $R$-module.
\begin{enumerate}
\item Considering $M$ as $A$-module, we have 
  $\Ann_A(M) = \Ann_R(M) \cap A$. If 
  $M$  is  finite over $R$ then  
  $$\Supp_A(M) = V(\Ann_A(M))$$
and  hence $\Supp_A(M)$ is closed in $\Spec A$ and the natural map $\Supp_R(M) \to \Supp_A(M)$ is dominant. 
\item Let $R$ be Noetherian, $I$ an ideal in $R$, $M$   finite over $R$, and $M^\wedge$ the $I$-adic completion of $M$. 
Then 
$$M^\wedge=0 \iff IM = M \iff I^n M = M \text{ for all } n\ge 1,$$
and this is equivalent to $\Supp_R(M)\cap V(I)= \emptyset$.
\end{enumerate}
\end{lemma}

\begin{proof}
1. Let $\fp \in \Spec A$. Note that $\fp \notin \Supp_AM \Leftrightarrow M_\fp=0 \Leftrightarrow \forall m\in M \  \exists f \in A$, $f \! \notin \! \fp$,  $fm=0$ and that 
$\fp \notin V(\Ann_A(M)) \Leftrightarrow  \Ann_A(M) \nsubseteq \fp   \Leftrightarrow   \exists f \! \notin \! \fp, f M= 0 $. Hence $\fp \notin V(\Ann_A(M))$ implies $\fp \notin \Supp_A(M)$, i.e. $\Supp_A (M) \subset V(\Ann_A(M))$. 
Now let $M$ be generated over $R$ by $m_1,\dots,m_k \in M$. If $M_\fp=0$, choose  $f_i \in A, f_i \! \notin \! \fp,  f_im_i=0.$ 
Then 
$f=f_1\cdots f_k \! \notin \! \fp$ satisfies $fM=0 $ and hence $f\in \Ann_A(M)$. We get $\fp \notin V(\Ann_A(M))$ and hence the other inclusion $\Supp_A (M) \supset V(\Ann_A(M))$.
Then  $A/\Ann_A(M) \to R/\Ann_R(M)$ is injective and hence $\Supp_R(M) \to \Supp_A(M)$ is dominant (\cite[Lemma 28.8.6, tag 0CC1]{Sta19}).

2. $ IM = M$ is, by induction on $n$,  equivalent to  $I^n M = M$ for all $n\ge 1$.
Since $\cap_{n=1}^\infty I^nM$ is the kernel of $M \to M^\wedge$ (Lemma \ref{lem.compl}.9.),
$M^\wedge =0$ implies $M=I^n M$ for all $n$.\\
Conversely, let  $M=I M$. Then, since 
$M  / I M=M ^\wedge / I M^\wedge$ by Lemma  \ref{lem.compl}.3. we get
$M ^\wedge / I M^\wedge =0$. Since  $M ^\wedge$ is an $R^\wedge$-module and $R^\wedge$ is $I$-adically complete, $I\subset \rad (R^\wedge)$ by Lemma \ref{lem.compl}.4. Hence
$M^\wedge=IM^\wedge$ implies $M^\wedge=0$ by  Lemma \ref{lem.compl}.10. (Nakayama).\\
The last statement follows since $M^\wedge =0  \iff M/IM=0
\iff \Supp_R (M/IM) = \Supp_R(M)\cap V(I) = \emptyset$.
\end{proof}

\begin{lemma} \label{lem.ann1}
Let $R$ be Noetherian, $I\subset R$ an ideal, $M$ a finitely generated $R$-module and $M^\wedge$ the $I$-adic completion of $M$. For $\fn \in \Spec R$ the following holds.
$$(M_\fn)^{\wedge} = 0 \iff \fn \notin \Supp_R(M)\cap V(I)
=\Supp_R(M/IM).$$
\end{lemma}

\begin{proof}
By Lemma \ref{lem.ann}.2. we have
$(M_\fn)^\wedge=(M_\fn)^{\wedge, IR_\fn}=0 \iff IM_\fn = M_\fn$.
Since $M_\fn$ is finitely generated over the local  ring $R_\fn$, we get from Nakayama's Lemma
$$IM_\fn = M_\fn \iff M_\fn = 0 \text{ or } IR_\fn =R_\fn
\iff \fn \notin \Supp_R(M) \text{ or }  \fn \notin V(I)$$
which is equivalent to  $\fn \notin \Supp_R(M) \cap V(I)$.
\end {proof}

\begin{lemma}\label{lem.open} Let $A\to R$ be a morphism of Noetherian rings, $M$ a finitely generated $R$-module and $I\subset R$ an ideal such that $R/I$ is finite over $A$.
\begin{enumerate}
\item We have (possibly strict) inclusions 
$$\{ \fp \in \Spec A \ | \ \hat d_\fp(M) \neq 0\} \subset \{\fp \,|\,d_\fp(M)\neq0\}  \subset \Supp_A (M),$$
where the first and the last  sets are closed in $\Spec A$, while the middle set may not be closed.
\item Let $\fp \in \Spec A$. Then  
$$\hat M(\fp)=M(\fp)^{\wedge} = 0 \iff \fp \notin \Supp_A(M/IM).$$
 Hence $\{ \fp \in \Spec A \ | \ \hat d_\fp (M) =0\}$, the {\em vanishing locus} of $\hat d_\fp (M)$, is open in $\Spec A$. 
Moreover, the natural map
$$
  \Supp_{R}(M/IM)= V(I) \cap \Supp_{R}(M)
 \longrightarrow  
  \bigl\{\fp \in \Spec A \mid \hat{d}_{\fp}(M) \neq 0\bigr\},
$$
$  \fq \mapsto \fq \cap A$,
is finite and surjective.
\end{enumerate}
 \end{lemma}

\begin{proof} 1. Both inclusions follows from the definitions. For an example where these inclusions are strict, see
Example \ref{ex.sem0} (2) and (3). In Example \ref{ex.sem0} (3) the middle set is not closed. The closedness of the first set follows from 2.

2. From Lemma \ref{lem.ann}.2 we deduce
 $M(\fp)^{\wedge} = 0  \iff M(\fp)/I M(\fp)= 0$.
Since $M(\fp)/I M(\fp) 
  \cong (M/IM) \otimes_A k(\fp)$ and since $M/IM$ is finite over $A$, $(M/IM)_\fp$ is finite over $A_\fp$.
  We get from Nakayama's lemma that 
 $ (M/IM) \otimes_A k(\fp) =0  \iff (M/IM)_\fp=0 \iff 
  \fp \notin \Supp_A(M/IM)$.
Since
$\Supp_A (M/IM)$
 is closed in $\Spec A$, the  vanishing locus of $\hat d_\fp (M)$ is open in $\Spec A$.
 
  Moreover, we get from Lemma \ref{lem.ann}.1 that the  map 
  $$ V(I) \cap \Supp_{R}(M)=\Supp_{R} (M/IM) \to \Supp_A (M/IM)$$
   is dominant. It is finite since $\Spec (R/I)\to \Spec (A)$ is finite. It is surjective, since finite maps are closed (\cite[tag 01WM]{Sta19}). 
\end{proof}  

Lemma \ref{lem.open} is used to prove the following general semicontinuity result.

\begin{theorem}[Semicontinuity theorem]\label{thm.scont}
Let  $A \to R$ be a morphism of Noetherian rings,  $I\subset R$ an ideal such that $R/I$ is finite over $A$, and $M$ a finitely generated $R$-module.
\begin{enumerate}
\item  Then $\hat d_\fp(M) = \dim_{k(\fp)} \hat M(\fp)$ is semicontinuous on 
$\Spec A$. 
\item If $M$ is flat over $A$, then $\hat d_\fp(M)$ is locally constant on $\Spec A$.
\end{enumerate}
\end{theorem}

Semicontinuity means that $\hat d_\fp(M) \ge \hat d_\fq(M)$ for $\fq$  in some open neighbourhood of $\fp$ in $\Spec A$. Equivalently, the set
$\{\fp \in \Spec A| \ \hat d_\fp(M)  \leq N\}$ is open in  $\Spec A$  for  $N=0,1,2,..., \infty.$ Example \ref{ex.notsemcont} shows, that semicontinuity of the usual fiber dimension $d_\fp(M)$ does in general not hold.

\begin{proof} 
Since $M$ is finite over $R$, $I^nM$ is finite over $R$ and $I^nM/I^{n+1}M$ is finite over $R/I$ for each $n\ge 0$.
Since $R/I$ is finite over $A$, $I^nM/I^{n+1}M$ is finite over $A$ for each $n\ge 0$. From the exact sequence
  $$0 \to I^nM / I^{n+1} M\to M/ I^{n+1} M\to M/ I^nM \to 0$$
it follows by induction that $M/ I^n M$ is finite over $A$ for all $n\ge1$. 
In the same way $I^n/ I^{n+1}$ and $R/ I^n$ are finite over $A$ for all $n\ge1$.

Fix $\fp \in \Spec A$. 
We have (by Lemma \ref{lem.compl}.6.)
that $ \hat M (\fp) = M (\fp)^{\wedge}$  is a finite 
$ \hat R (\fp)$-module.
To prove semicontinuity of $\dim_{k(\fp)}\hat M(\fp)$ at $\fp$, we may assume  that $\hat M (\fp)\ne 0$ (by Lemma \ref{lem.open}.2) and that $\dim_{k(\fp)}\hat M(\fp)< \infty$ (otherwise semicontinuity at $\fp$ is trivial).
Then  $I\hat M (\fp)$ is strictly contained in $\hat M (\fp)$ by Nakayama's lemma (Lemma \ref{lem.compl}.10.),
since $I\hat R(\fp)$ is contained in the Jacobson radical of  $\hat R (\fp)$
by Proposition \ref{prop.sem}.3.
This implies 
$I^N \hat M(\fp) =
 (I^N M) \hat \oplus_A k(\fp) = 0 $ for some  $N =N(\fp) \le \dim_{k(\fp)} \hat M(\fp)$ and hence 
$(I^{n} M) \hat \oplus_A k(\fp) =0$ for $n\ge N$. Again by Lemma \ref{lem.open}.2 it follows that $(I^{n} M) \hat \oplus_A k(\fq) =I^n \hat M(\fq)=0$ for $n\ge N$ and for $\fq$ in some open neighbourhood $U$ of $\fp$.

Since $M/I^nM = \hat M/I^n\hat M$ is finite over $A$ for $n\ge 1$, 
$$\dim_{k(\fp)}(\hat M/I^{n}\hat M) \otimes_{A} k(\fp) =
\dim_{k(\fp)}\hat M(\fp)/I^{n}\hat M(\fp)$$ is semicontinuous on $\Spec A$ by Lemma \ref{lem.free}. We use now that $I^{n}\hat M(\fq) =0$ for $n\ge N$ and $\fq \in U$ to deduce  
 that  $\dim_{k(\fq)}\hat M(\fq) = \dim_{k(\fq)}\hat M(\fq)/I^{n}\hat M(\fq)$ is semicontinuous on $U$. 
 Since $\fp \in \Spec A$ was arbitrary, the semicontinuity of $\dim_{k(\fq)}\hat M(\fq)$ on $\Spec A$ follows. 
 
  If $M$ is flat over $A$, then $\dim_{k(\fq)}\hat M(\fq)$ is locally constant on  $\Spec A$, again by Lemma \ref{lem.free}.
\end{proof}

\begin{remark}\label{rm.cont}
{\em
(1) For $\fp \in \Spec A$ with $\hat R(\fp) \ne 0$, the completed fiber $\hat R(\fp)$ is a
semi-local ring
with maximal ideals $\fm_1,...\fm_r$, isomorphic to   
$\hat R(\fp)_{\fm_1}\oplus\ldots\oplus \hat R(\fp)_{\fm_r}$  (Proposition \ref{prop.sem}).
 $\hat R(\fp)_{\fm_i}$ is a complete local ring with residue field
$k(\fm_i) = \hat R(\fp)/\fm_i$, which is finite over $k(\fp)$.
Moreover, 
$\hat M(\fp)\cong \hat M(\fp)_{\fm_1}\oplus\ldots\oplus \hat M(\fp)_{\fm_r}$ and hence 
$$\dim_{k(\fp)}  \hat M(\fp)=
\sum_{i=1}^r \dim_{k(\fp)} \hat M(\fp)_{\fm_i}=
\sum_{i=1}^r [k(\fm_i):k(\fp)]\dim_{k(\fm_i)} \hat M(\fp)_{\fm_i}.$$ 

(2) Assume that $\fp$ is a {\em maximal}  ideal of $A$. Let $\fm_i$ be one of the maximal ideals of 
$\hat R(\fp)$ and $\tilde \fn_i =\fm_i \cap R$ the preimage of $\fm_i$ in $R$. Then, according to Proposition \ref{prop.sem2}, $\tilde \fn_i$ is a maximal ideal and there  exists   
an open neighbourhood $U_i = \Spec Q_i \subset \Spec R$  of $\tilde\fn_i$  such that $\tilde \fn_j \notin U_i$ if $j\ne i$. 
There exists also an  open neigbourhood  $\Spec A_i \subset \Spec A$ of $\fp$ 
such that $Q_i/IQ_i$ is finite over $A_i$. 
Moreover,
$\fm_i \hat Q_i(\fp)$ is the only maximal ideal of $\hat Q_i(\fp)$. For $\fq \in \Spec A_i$ the ring $\hat Q_i(\fq)$ is semi-local with maximal ideals, say, $\fn_{i,1},...,\fn_{i,t}$,
and $\hat M_{Q_i} := M\hat \otimes_R Q_i$ satisfies $\hat M_{Q_i} (\fp) = \hat M(\fp)_{\fm_i}$. We have
$$\dim_{k(\fp)} \hat M(\fp)_{\fm_i} \geq
\sum_{j=1}^t \dim_{k(\fq)} \hat M_{Q_i}(\fq)_{\fn_{i,j}}
 \ \text {for all } \fq \in U_i.$$
}
\end{remark}

In  \cite{GP21} we considered families $A \to R$ of complete local rings along the trivial section. That is, $R= A[[x]]$ is the completion of $A[x]$ along the ideal $I = \x = \langle x_1,...,x_n\rangle$. 
 The following corollary was proved in \cite[Theorem 2.1]{GP21} for $\dim A =1$ and for arbitrary Noetherian $A$ under the assumption that $M$ is algebraically presented in \cite[Theorem 2.2]{GP21} (and a few other cases 
 \cite[Corollary 2.7]{GP21}).  These restrictions can now be removed.
\begin{corollary} \label{cor.scont}
Let $A$ be a Noetherian ring, $R = A[[x]]=A[x]^{\x}$, $x=(x_1,\ldots,x_n),$ and $M$ a finitely generated $R$-module.
 Then, for every $\fp \in \Spec A$, 
the ideal $\fn_\fp :=  \langle \fp, x_1,...,x_n\rangle_{R} $ is a prime ideal of $R$ and 
 the completed fiber w.r.t. $I=\x$ is the local ring  $\hat R(\fp) \cong (R_{\fn_\fp}/\fp R_{\fn_\fp})^\wedge\cong k(\fp)[[x]] $. Moreover, $\hat M(\fp) \cong (M_{\fn_\fp}/\fp M_{\fn_\fp})^\wedge$ and the function $\fp \mapsto \dim _{k(\fp)} \hat M(\fp)$ is semicontinuous on $\Spec A$.
\end{corollary}

\begin{proof} 
We note that $\fn_\fp R_\fp$  
is a maximal ideal of $R_\fp$ and that
 $R(\fp)^\wedge  \cong (R_{\fn_\fp}/\fp R_{\fn_\fp})^\wedge$.
This  follows from Proposition \ref{prop.sem}.2. and  \ref{prop.sem}.3..
 The semicontinuity follows from Theorem \ref{thm.scont}.
\end{proof}

\section{Smoothness and singular locus}\label{sec.sing}

Before  we apply the semicontinuity theorem of the previous section,
 let us recall some facts about singularities and smoothness in our situation. 
 
\begin{definition} \label{def.sing}
Let $A$ be a Noetherian ring.
\begin{enumerate}
\item If  $A$ is local with maximal ideal $\fm$, then  $A$ is  called {\em regular} if 
$\dim A$ is equal to its embedding dimension $\dim_{A/\fm} \fm/\fm^2$. 
\item $A$ is called {\em regular} if the localization at every prime ideal (equivalently at every maximal ideal, cf. \cite[Chapter 7]{Mat86}) is a regular local ring. 
\item We define the {\em regular locus} of $A$ as
$$\Reg (A):= \{\fp \in \Spec A \mid A_\fp \text{ is regular}\}$$
 and the {\em singular locus} as
$$\Sing (A):=\{\fp \in \Spec A \mid A_\fp \text{ is not regular}\}.$$
\end{enumerate}
\end{definition}

\begin{remark}\label{rm.sing}
\em {
(1) A Noetherian local ring $(A,\fm)$ is regular iﬀ its completion is regular (Lemma \ref{lem.compl}.13). 
Moreover, if $(A,\fm)$ is regular, then $A_\fp$ is regular for all 
$\fp \in \Spec A$ (\cite[Theorem 19.3]{Mat86}).

(2) For arbitrary Noetherian rings $A$ the regular locus 
need not be open in $\Spec A$. However, 
$\Reg (A )$ is open if $A$ is a complete Noetherian local ring 
(\cite[Corollary of Theorem 30.10]{Mat86}) and then the  singular locus 
$\Sing(A)$ is closed in $\Spec A.$
}
\end{remark}

We also recall the notion of smoothness, which is related to regularity, but better suited for families.

\begin{definition}\label{def.smooth}
Let  $\phi:A \longrightarrow B$ be a morphism of rings and $I$ an ideal in $B$. $\phi$ is called {\em (formally) $I$-smooth}, or
$B$ is called {\em (formally) $I$-smooth over $A$}, if for any $A$-algebra $C$ and any 
continuous\footnote{ Here we consider $B$ with the $I$-adic topology and $C/N$ with the discrete topology; $u$ is continuous if $u(I^m)=0$ for some $m$.}  
$A$-algebra homomorphism $u:B \longrightarrow C/N$, $N$ an ideal in $C$ with $N^2=0$, such that $u(I^m)=0$ for some $m$, there exist a morphism $\sigma:B\longrightarrow C$ such that $\pi\sigma=u$:
\[
\xymatrix{B \ar[r]^-{u}\ar@{-->}[dr]^\sigma & C/N\\
A \ar[u]^\phi\ar[r]_v & C\ar[u]_\pi
}
\]
If $I=0$ then $B$ is called a {\em formally smooth} $A$-algebra.
\end{definition}

\begin{lemma}\label{lem.ismooth}
Let $A$ be a ring, $B$ an $A$-algebra and $I$ an ideal in $B$.
\begin{enumerate}
\item 
Let $\widehat B$ be the  $I$-adic completion of $B$. Then the following are equivalent:
\begin{enumerate}
\item $A\longrightarrow B$ is $I$-smooth, 
\item $A \longrightarrow \widehat B$ is $I\widehat B$-smooth.
\end{enumerate}
\item Let $I' \subset B$ an ideal, $I\subset I'$. If 
$A\longrightarrow B$ is $I$-smooth, it is $I'$-smooth.
\item Let $A\to A'$ be any ring map. If 
$A\longrightarrow B$ is $I$-smooth then 
$A'\longrightarrow B' = B\otimes_A A'$ is $I'=IB'$-smooth.\\
In particular,  if $A\to B$ is $I$-smooth and  $\fp \in \Spec A$, then $A_\fp \to B_\fp$  is $I B_\fp$-smooth, and $k(\fp) \to B(\fp)$ is $IB(\fp)$-smooth.
\end{enumerate}
\end{lemma}

\begin{proof} 1. is proved in  \cite[Lemma 3.1]{GP21} (see also \cite[tag 07ED]{Sta19}), for 2. see  \cite[tag 07EE]{Sta19} and for 3,  \cite[tag 07EG]{Sta19}.
\end{proof}

\begin{lemma}\label{lem.smooth}
Let $A$ be any ring.
\begin{enumerate}
\item $A[x]$ is $\x $-smooth over $A$. 
\item $A[[x]]$ is $\x $-smooth over $A$.
\end{enumerate}
\end{lemma}

\begin{proof}
(1) Define the dotted arrow simply by choosing lifts $c_i \in C$ of the elements in $C/N$ to which the elements $x_i$ map. \\
(2) Folllows from (1) and Lemma \ref{lem.ismooth}.1., since $A[[x]]$ is the $\x$-adic completion of $A[x]$.
\end{proof}

Recall that a Noetherian $K$-algebra, $K$ a field, is {\em geometrically regular over $K$}, if $A\otimes_K K'$ is a regular ring for every finite extension field $K'$ of $K$.

\begin{proposition}\label{prop.greg}
Let $K$ be a field and $A$ a Noetherian $K$-algebra. The following are equivalent:
\begin{enumerate}
\item A is geometrically regular over $K$.
\item $A\otimes_K K'$ is  regular for every finite purely inseparable  extension field $K'$ of $K$.
\item $A \otimes_K K'$ is regular for every extension field $K'$ of $K$.
\item $A_\fm$ is geometrically regular over $K$ for every maximal ideal (equivalently for every prime ideal)  $\fm$ of $A$.
\item $\hat A_\fm$, the $\fm$-adic completion of $A_\fm$,  is geometrically regular over $K$ for every maximal ideal 
(equivalently for every prime ideal) $\fm$ of $A$.
\item $A_\fm$ is  $\fm A_\fm$-smooth over $K$ for every maximal ideal (equivalently for every prime ideal) $\fm$ of $A$.
\end{enumerate}
\end {proposition}

\begin{proof}
We prove first the equivalence of 1. and 4.. 
 $A$ is geometrically regular iff $B:=A \otimes_K K'$ is regular for every finite field extension $K \subset K'$. 
 $B$ is regular iff $B_\fn$ is regular for every prime ideal (maximal ideal) $\fn$ of $B$ (cf. Definition \ref{def.sing}).
 Setting $\fp = A\cap \fn$, then $\fp B \subset \fn$ and $B_\fn = (B_\fp)_{\fn B_\fp}$. Since $B_\fp$ is regular iff 
 $(B_\fp)_{\fn'}$  is regular for every maximal ideal
 $\fn'$ of $B_\fp$, and
 since all maximal ideals of $B_\fp$ are of the form $\fn B_\fp$ for $\fn$ a prime ideal of $B$ with 
 $\fp = A\cap \fn$ (Lemma \ref{lem.sem0}.2.), we get
 $$
 \begin{array}{lll}
 &\iff B \text{ is regular for all finite field extensions } K \subset K' \\
 & \iff  B_\fp = A_\fp \otimes_K K' \text{ is regular for all } \fp \in \Spec A \\
 & \quad\quad \  \text{ and all finite field extensions } K \subset K' \\
& \iff A_\fp \text{ is geometrically regular for all } \fp \in \Spec A.
 \end{array}
 $$
 This shows $1. \iff 4.$ for prime ideals of $A$. We still have to show $4. \Longrightarrow 1.$ for maximal ideals of $A$.
Let $\fp$ be a maximal ideal of $A$, and $\tilde\fm_1,...,\tilde \fm_r$ the maximal ideals of $B$ with $\tilde\fm_i=\fp \cap B$ (Lemma \ref{lem.sem0}.3.).  Then $B_\fp$ regular iff 
$(B_\fp)_{\tilde\fm_\fp B_\fp}= B_{\tilde\fm_\fp}$ is regular for  all maximal ideals $\tilde\fm_\fp$ of $B$ lying over $\fp$.
Since $B_\fp$ is semi-local (Lemma \ref{lem.sem0}.3.) and  all maximal ideals of $B$ are of this form by Lemma \ref{lem.sem0}.3., $B$ is regular and hence $A$ is geometrically regular.

Since $A_\fm$
is local with maximal ideal $\fm A_\fm$ for every prime ideal 
$\fm$ of $A$, we may therefore assume that $A$ is local
for the equivalence of  statements 1.,2.,3. and for the equivalence of 4.,5.,6.. 
In the following let $A$ be local with maximal ideal $\fm$.

We show  the equivalence of 4. and 5.. Let $\hat A$ be the $\fm$-adic completion of $A$, and $K \subset K'$ a finite field extension. We set $B:=A \otimes_K K'$ and  $C:=\hat A \otimes_K K' = 
 (\varprojlim A/\fm^k) \otimes_K K' =
 (A \otimes_K K' )^{\wedge,\fm B} = B^{\wedge,\fm B}$ (by Lemma \ref{lem.lim}).
Then $A \to B$ resp. $\hat A \to C$ are finite morphisms with fiber
$B(\fm) = B/\fm B=C/\fm C=C(\fm)$ being a semi-local ring with maximal ideals
$\bar\fm_1,...,\bar\fm_r$ by Lemma \ref{lem.sem0}.1.. 
By Lemma \ref{lem.sem0}.2.
$B$ resp. $C$ are semi-local rings with maximal ideals the  preimages $\fn_{i}=\bar\fm_i \cap B$ resp.  $\tilde\fm_{i}=\bar\fm_i \cap C$ in $B$ resp. $C$, $i=1,...,r$.\\
Now we use Proposition \ref{prop.sem} with $I=0$, $\fp = \fm$, $R=R_\fp = B$, and $\tilde R_\fp = C$. Then 
Proposition \ref{prop.sem}.2. implies 
$C= C^{\wedge,\tilde \fm_1}\oplus ... \oplus C^{\wedge,\tilde \fm_r}$, with $C_{\tilde \fm_i} = C^{\wedge,\tilde \fm_i} = B^{\wedge,\fn_i} =(B_{\fn_i})^{\wedge,\fn_i}$
complete local rings. We have that 
$$
\begin{array}{llll}
&B \text{ is regular } \iff 
B_{\fn_i} \text{ is regular for all  $i$}  \iff (B_{\fn_i})^{\wedge, \fn_i} \text{ is regular for all $i$},\\
&C\text{ is regular }  \iff 
C_{\tilde\fm_i} = (B_{\fn_i})^{\wedge, \fn_i} \text{ is regular for all  $i$}. 
\end{array}
$$
Hence $B$ is regular $\iff$ $C$ is regular. Since the
field extension $K \subset K'$ is arbitrary, it follows that
$B$ is geometrically regular iff $C$ is geometrically regular, which proves $4. \iff 5.$.

It remains to prove the equivalence of statements 1.,2.,3.,6.
with $(A,\fm)$ local.
The equivalence of 1. and 2. is proved in \cite[tag 0381]{Sta19}, the equivalence of 1. and 6. in \cite[Theorem 28.7 ]{Mat86}. 

Since 3. $\Rightarrow$ 1. is trivial, let us prove 1. 
$\Rightarrow$ 3.
 Assume 1. and let $K'$ be any field extension of $K$. 
By 1. $\Rightarrow$ 6. $A$ is  $\fm$-smooth over $K$ and 
from Lemma \ref{lem.ismooth}.3. we get that the (semi-local) $K'$-algebra
$B=:A \otimes_K K'$ is $\fm B$-smooth over $K'$.  
Again by Lemma \ref{lem.ismooth}.3. the local ring
$B_\fn$ is $\fm B_\fn$-smooth
for every prime ideal (maximal ideal) $\fn$ of $B$. By Lemma \ref{lem.sem0}.2. $\fm B \subset \fn$ and by Lemma \ref{lem.ismooth}.2. we have that $B_\fn$ is $\fn B_\fn$-smooth over $K'$. By
6. $\Rightarrow$ 1. implies that
$B_\fn$ is geometrically regular over $K'$, in particular a regular ring. Hence $B=A\otimes_K K'$ is regular, showing 3..
\end{proof}

\begin{remark} \label{rm.reg}
Let $(A,\fm)$ be a local Noetherian ring containing a field $K$. 
\begin{enumerate}
\item Let $A/\fm$ be separable over $K$. Then 
$$A \text{ is  $\fm$-smooth over  $K$}  \Leftrightarrow \text{$A$ is geometrically regular over $K$} \Leftrightarrow  \text{$A$ is regular.}$$
\item Let $K'$ be a finite field extension of $K$. Then $A'=A\otimes_K K'$ is a semi-local ring.
\end{enumerate}
\end{remark}
\begin{proof}
 1. follows from \cite[Theorem 28.7 and Lemma 1, page 216]{Mat86}.\\
2. Since $A\to A'$ is finite , $A'= A'_\fm$ is semi-local by Lemma \ref{lem.sem0}.2.
\end{proof}

\begin{example}{\em
We show that finiteness of the field extension is necessary.
Let $t$ be a variable, $S=\Q[t]\setminus \{0\}$ then $\Q(t)\otimes_{\Q} \C=S^{-1}\C[t]$. $\langle t-\alpha\rangle \cap S=\emptyset$ iff $\alpha$ is transzendental over $\Q$.  This implies that $\Q(t)\otimes_{\Q} \C$ has uncountably many maximal ideals, in particular, it is not semi-local.}
\end{example}

\begin{lemma}\label{lem.reg.fieldextension}
Let $K$ be a field and $K'$ be a separable algebraic field extension of $K$. Let $A$ be a $K'$-algebra.
$A$ is geometrically regular over $K$ if and only if $A$ is geometrically regular over $K'$.
\end{lemma}
\begin{proof}
The lemma is proved in \cite[tag 07QH]{Sta19}.
\end{proof}

\begin{example}\label{ex.reg.fieldextension}{\em
We show that the separability of the field extension is necessary in the above lemma.
Let $K$ be a field of characteristic $p>0$. Let $K:=K(t^p)\subset  K(t)=:K'$ for an element $t$ transcendental over $K$, which is a non-separable extension. 
Let $z:=t\otimes 1 - 1\otimes t\in K'\otimes_KK'$. Then $z\neq 0$ but $z^p=0$. Let $A:=K'[[x]]$ then we obtain from lemma \ref{lem.smooth} and proposition \ref{prop.greg} that $A$ is geometrically regular over $K'$. Now the canonical map $K'\otimes_KK'\longrightarrow K'\otimes_K A$ is injective since the canonical injection $i:K'\longrightarrow A$ is a split map\footnote{The projection $p:A\longrightarrow K'$ defined by $p(\sum_{i\geq 0}a_ix^i)=a_0$ is $K$-linear and $p\circ i=id_{K'}$. The corresponding projection induced by the tensor product $p'=id_{K'}\otimes p :K'\otimes_K A\longrightarrow k'\otimes_KK'$ is also $K$-linear and $p'\circ i'=id_{K'\otimes_KK'}$ with $i'=id_{K'}\otimes i$.}. This implies that $z$ considered as element in $K'\otimes_K A$ is a non-zero nilpotent element and therefore $A$ is not geometrically regular over $K$.}
\end{example}

By the following important theorem of Grothendieck (\cite{Mat86} Theorem 28.9), smoothness of a morphism is equivalent to flatness and smoothness of the fibers.

\begin{theorem}\label{thm.Groth}
Let $(A,\fm) \to (B,\fn)$ be a morphism of Noetherian local rings. Let $\bar B=B/\fm B$ and $\bar \fn=\fn/\fm B$. Then $B$ is $\fn$-smooth over $A$ iff $\bar B$ is $\bar \fn$-smooth (geometrically regular) over $A/\fm$ and $B$ is flat over $A$.
\end{theorem}

\medskip

Now let  $K[[x]], \,  x = (x_1,\cdots, x_n),$ be the formal power series ring over the field $K$ and  $J$ an ideal in  $\x K[[x]]$. If $J$ is generated by  $f_1,\ldots,f_m$ we denote by $Jac(J)\in Mat(m,n,K[[x]])$ the Jacobian matrix $(\partial f_j / \partial x_i)$ and by  
$I_k(Jac (J))$ the ideal generated by the $(n-k)\times (n-k)$-minors of $Jac(J)$. This is the $k$-th Fitting ideal of the cokernel of the transpose of $Jac(J)$, considered as map 
$(\partial f_j / \partial x_i)^t: K[[x]]^m \to K[[x]]^n$, and thus independent of the chosen generators $f_j$.\footnote{$I_n(Jac(I))$, the $n$-th Fitting ideal of $K[[x]]^n$, is defined as $K[[x]]$.}\\

For analytic local rings the Jacobian criterion characterizes regularity:

 \begin{lemma} \label{lem.jac}
Let $K$ be any field and $J\subset \x$ an ideal of $K[[x]]$.  If $\dim K[[x]]/J=d$, the following are equivalent.
\begin{enumerate}
\item $K[[x]]/J$ is $\x$-smooth over $K$.
\item $K[[x]]/J$ is a regular local ring. 
\item $K[[x]]/J$ is geometrically regular over $K$.
\item $K[[x]]/J \cong K[[y_1,\dots,y_d]].$ 
\item $I_d(Jac (J)) = K[[x]]$ {\em (Jacobian criterion)}.\\
In particular, if $J=\langle f\rangle$ is a principal ideal then
$K[[x]]/\langle f\rangle$ is regular iff the partial derivatives of $f$ generate $ K[[x]]$.
\end{enumerate}
\end{lemma}

\begin{proof}
The equivalence of 1., 2. and 3. follow from Remark \ref{rm.reg}.2.
($K$ coincides with the residue field of $K[[x]]/J$),
the equivalence of 2. and 4.  from Cohen's structure theorem (Theorem \ref{thm.cohen}) below. 

 For the Jacobian criterion let $J=\langle f_1,...,f_m \rangle$
 and we note that $n-d\le m$ (by \cite[Theorem B.2.1.]{GLS07}).
That  4. implies  5. follows from the definition of $I_n(Jac (J))$, while 
5. implies 4. is a consequence of the implicit mapping theorem for formal power series:
$I_d(Jac (J)) = K[[x]]$ is equivalent to $I_d(Jac (J))$ containing
a unit of $K[[x]]$, which is equivalent to the existence of an $(n-d)$-minor of $Jac(J)|_{x=0}$, which is non-zero in $K$. By \cite[Theorem I.1.18]{GLS07} we get $K[[x]]/J\cong K[[y_1,\dots,y_d]]$. 
\end {proof}

\begin{theorem}[Cohen's structure theorem]\label{thm.cohen}
Let $(R,\fm)$ be a Noetherian complete local ring, containing a field\,\footnote{Also called the equal characteristic case. Then $R$ contains a {\em coefficient field  K},  i.e. $K$ maps isomorphically to $k(\fm)$
under $R\to R/\fm$.} and $k(\fm) = R/\fm$ its residue field. Then
$$R \cong k(\fm)[[x_1,...,x_n]]/J,$$
for some ideal $J$. If $\fm = \langle m_1,...,m_n\rangle$ and
$\phi : k(\fm)[[x_1,...,x_n]] \to R$ the map $x_i \mapsto m_i$, then $\phi$ is surjective and $J = \ker(\phi)$.\\
If $R$ is regular of dimension $d$, then 
$$R \cong k(\fm) [[x_1,...,x_d]].$$
For a proof see see \cite[Lemma 10.155.10, tag 0C0S]{Sta19} and \cite[Theorem 6.3.10]{GP08}.
\end{theorem}

Using the Jacobian criterion, we can describe the singular locus (see Definition \ref{def.sing}) for analytic algebras in a form which is very useful. 

\begin{lemma}\label{lem.slocus}
Let $B:= K[[x]]/J$, $J$ an ideal in $\x K[[x]]$.
\begin{enumerate}
\item If  $B$ is pure $d$-dimensional (i.e. $\dim B/P =d$ for all minimal primes $P \in \Spec B$), then
$$\Sing (B) = V(I_d(Jac(J))\cdot B) \subset \Spec B,$$
which we identify with $V(I+I_d(Jac(J)) \subset \Spec K[[x]].$
\item If $B$ is not pure dimensional, let $P_1,\dots,P_r$ be the minimal primes  of $B$. Then $B/P_i$ is pure dimensional and then  
$$\Sing (B) = \bigcup_{i=1}^r \Sing (B/P_i) \ \cup  \ \bigcup_{i \ne j}  V(P_i)\cap V(P_j)\subset \Spec B.$$
\end{enumerate}
\end{lemma}

In the first case, if we identify $V(J) \subset \Spec K[[x]]$ with 
$\Spec B$,  then $\Sing(B) = V(J+I_d(Jac(J))$ is a subscheme of $\Spec K[[x]]$ and carries a natural scheme structure given by the ideal $J+I_d(Jac(J)) \subset K[[x]]$.  In general we endow  $\Sing(B)$ with its reduced structure,  which is given by the ideal 
$I(\Sing(B)) = \bigcap_{P\in \Sing (B)} P$.

\begin{definition}\label{def.iso}
We say that  $K[[x]]/J$ (or $J$) has an {\em isolated singularity} (at 0) if the maximal ideal  $\x$ is an isolated point of $\Sing(K[[x]]/J)$ or if $K[[x]]/J$ is regular, i.e. 
 $\Sing(K[[x]]/J) =\{\x\}$ or $\emptyset$.
\end{definition}

\begin{proof} [Proof of the lemma] Let $P\in \Spec B$ and $B_P$ regular of dimension $d$. Then the
completion $\hat B_P$ is regular of dimension $d$ (Remark \ref{rm.sing}) and  $\hat B_P \cong k(P)[[x_1,...,x_d]]$ by Theorem \ref{thm.cohen}.
From Lemma \ref{lem.jac}.3 we see that $\hat B_P$ and hence  $B_P= K[[x]]_P/J_P$ is irreducible and from Lemma \ref{lem.jac}.4 that $P \notin V(I_d(Jac(J)))$.

Conversely, let $P\notin \bigcup_{i=1}^r \Sing (B/P_i) \ \cup  \ \bigcup_{i \ne j}  V(P_i)\cap V(P_j)$. Then there exists a minmal prime $P_i$ of $B$ such that $P \in V(P_i)$ and 
$P \notin V(P_j)$ for $j\ne i$. Then $B_P$ is
pure $d$-dimensional, $d= \dim B/P_i$, and $P\notin \Sing (B/P_i)$, i.e. $P\notin V(I_d(K[[x]]/P_i))$. Lemma \ref{lem.jac} implies that $K[[x]]/P_i$ is regular. Hence $B_P = (K[[x]]/P_i)_P$ is regular by Remark \ref{rm.sing}.
\end{proof}
 
\medskip

 Let us now consider families. 
 
 \begin{definition}\label{def.smooth1}
 Let $\phi: A \to R$ be a morphism of Noetherian rings.  $\phi$ is called {\em flat} (resp. {\em smooth}) at $ \fn\in \Spec R$,  if the induced morphism of local rings,
 $\phi_{\fn}: A_\fp \to R_\fn$,  $\fp := A\cap \fn$,
 is flat (resp. $ \fn R_{ \fn}$-smooth). $\phi$ is called  flat,  (resp. smooth) if this holds at every $\fn \in \Spec R$. 
We denote the {\em smooth locus} of $\phi$ by
 $$\Sm(\phi) := \{\fn \in \Spec R \mid \phi \text{ is smooth at } \fn\}$$
 and by 
 $$\NSm(\phi) := \{\fn \in \Spec R \mid \phi \text{ is not smooth at } \fn\}$$
 the {\em non-smooth locus} of $\phi$.
\end{definition}

\begin{remark}\label{rm.sopen}{\em
(1)  By Grothendieck (Theorem \ref{thm.Groth}) $\phi$ is smooth at $\fn$ iff 
$\phi$ is flat at $\fn$ and  if the fiber $R(\fp)_\fn = R_\fn(\fp)=R_\fn\otimes_{A_\fp}k(\fp)$ of $\phi_{\fn}$ is  $\fn R_\fn(\fp)$-smooth (equivalently geometrically regular, by Proposition \ref{prop.greg}) over $k(\fp)$. 

(2) In general the smooth locus may not be open, but if $\phi: A \to R$ is of finite presentation  (i.e. $R \cong A[x_1, \ldots, x_n]/\langle f_1, \ldots, f_m\rangle$)
then $\Sm (\phi)$ is open in $\Spec R$. This follows since the flat locus as well as the locus, where the fiber is geometrically regular, are open by by \cite[IV$_3$, Theorem 11.1.1]{EGA} and  \cite[IV$_3$, Theorem 12.1.6]{EGA}. }
\end {remark}

\begin{notation}\label{rm.notation}{\em
Now  let $I \subset R$ be an ideal such that $R/I$ is finite over $A$. 
 For $\fp \in \Spec A$ the completed fiber
$(R_\fp/\fp R_\fp)^{\wedge,I}=\hat R(\fp)$ is a semi-local $k(\fp)$-algebra by Proposition \ref{prop.sem}. To make the dependence on $\fp$ visible, we denote its maximal ideals by 
$\fm_{\fp,1},...,\fm_{\fp,r_\fp}$ instead of $\fm_{1},...,\fm_{r}$
(and similar with the ideals $\bar \fm_i$ and the other ideals derived from it).
Then
 $$\hat R(\fp) \cong   \hat R(\fp)_{\fm_{\fp,1}}\oplus\ldots\oplus \hat R(\fp)_{\fm_{\fp,r_\fp}}$$
 is semi-local  with $\hat R(\fp)_{\fm_{\fp,i}}$ a complete local ring for each $i=1,...,r_\fp$, which contains the residue field $k(\fp)$ of $A_\fp$. By Theorem \ref{thm.cohen}
$$\hat R(\fp)_{\fm_{\fp,i}} \cong k(\fm_{\fp,i})[[x_1,...,x_n]]/J_{\fp,i}$$
  for some ideal $J_{\fp,i}$, and the residue field $k(\fm_{\fp,i})$ of $\hat R(\fp)_{\fm_{\fp,i}}$ is  finite over $k(\fp)$, but not necessarily separable as the following example shows.}
 \end {notation}
 
\begin{example} \label{ex.nsep}{\em
Let $K$ be a field of characteristic $p>0$, $A=K[t]$, $P=A[x]$,
$I=\langle x^p -t\rangle$ and $R=P^{\wedge,I}$. 
For $\fp=\langle 0 \rangle \in \Spec A$ we have
$$ \hat R(\fp) = K(t)[x]^{\wedge,I}. $$
This is a complete local ring with maximal ideal 
$\fm = \langle x^p -t\rangle  \hat R(\fp)$ (by Proposition \ref{prop.sem}) and residue field 
$k(\fm) =  K(t)[x]/ \langle x^p -t\rangle = K(\sqrt[p]{t})$, which is not separable over $K(t)$. By Cohen we have $ \hat R(\fp) \cong  K(\sqrt[p]{t})[[x]]$.}
\end{example}

\section{Applications to singularities}\label{sec.appl}
We want to apply  Theorem \ref{thm.scont} to prove the semicontinuity of several singularity invariants in families parametrized by an arbitrary Noetherian ring.
Let us  first prove some results about flatness and regular sequences that will be used when we consider complete intersection singularities.
\medskip
 
Let $F_1,...,F_m$ be elements of a ring $R$ and $M$ an $R$-module. The sequence $F_1,...,F_m$ is called an {\em $M$-regular sequence} or an {\em $M$-sequence} if, for $i=1,...,m$,
$F_i$ is a non-zero divisor of $M/\langle F_1,..., F_{i-1}\rangle M$ and if $M/\langle F_1,..., F_m\rangle M\ne 0$.

\begin{proposition}\label{prop.reg}
Let $A\to R$ be a morphism of Noetherian rings,  
$I\subset R$ an ideal,
$F_1,...,F_m$ elements of $R$, and $M$ a finitely generated $R$-module. 
Let  ${}^\wedge$ denote  the $I$-adic completion. 
For $\fp \in \Spec A$ let 
$F_i(\fp)$ resp. $\hat F_i(\fp)$ denote the image of $F_i$ 
in $R(\fp) = R\otimes_A k(\fp)$ resp. in $\hat R(\fp) = R(\fp)^{\wedge}$.

 \begin{enumerate}
\item [(i)] 
The following are equivalent.
 \begin{enumerate}
 \item   
 $M/ \langle F_1,..., F_m\rangle M$  is flat over $A$ and $F_1,\ldots,F_m$  is an $M$-regular sequence.
 \item $M$ is flat over $A$ and for every $\fp \in \Supp_A(M)$, $F_1(\fp),\ldots,F_m(\fp)$  is an $M(\fp)$-regular sequence.
\end{enumerate}
Moreover, condition (a) implies that $M/\langle F_1,...,F_i\rangle M$, $i=0,...,m$, is flat over $A$.
\item [(ii)] 
The following conditions {\em (a)} and  {\em (b)} are equivalent and imply  {\em (c)}.
\begin{enumerate}
\item$M^\wedge/ \langle F_1,..., F_m\rangle M^\wedge$ is flat over $A$ and $F_1,\ldots,F_m$  is an $M^\wedge$-regular sequence.
\item $M^\wedge$ is flat over $A$ and and for every $\fp \in \Supp_A(M^\wedge)$,  $F_1(\fp),\ldots,F_m(\fp)$  is an $M^\wedge(\fp)$-regular sequence. 
\item  $M^\wedge$ is flat over $A$ and and for every $\fp \in \Supp_A(M^\wedge)$,  $\hat F_1(\fp),\ldots,\hat F_m(\fp)$  is a 
$\hat M(\fp)$-regular sequence. 
 \end{enumerate}
 Moreover, statement {\em (a)} of {\em (i)} implies statement  {\em (a)} of {\em (ii)}. 
 \smallskip
 
 \item [(iii)] 
 Let  $I\subset R$ be an ideal with $R/I$ finite over $A$
 and $F_{1},\ldots,F_{m} \in I$. Let $M$ and $M/\langle F_{1},\ldots,F_{m}\rangle M$ be $A$-flat and let the induced map $\Supp_R(M/IM) \to \Spec A$ be open.\footnote{E.g. if $M/IM$ is flat over $A$. Since $A\to R/I$ is finite,
 $\Supp_R(M/IM) \to \Spec A$ being open is equivalent to $A\to R/\Ann_R(M/IM)$ has the going-down property.}
For $\fp \in \Supp_A(M/IM)$ let
  $\hat F_1(\fp),\ldots,\hat F_m(\fp)$  be an $\hat M(\fp)$-sequence. 
Then there exists an  open neighbourhood $U$ of \ $\fp$  in $\Spec A$ such that 
 $\hat F_1(\fq),\ldots,\hat F_m(\fq)$ is a $\hat M(\fq)$-sequence for all $\fq$ in $U$.
\end{enumerate}
\end{proposition}

\begin{proof}
(i) Set $M_0=M,  M_i=M/\langle F_{1},\ldots,F_{i}\rangle M$ and  consider for 
$i=1,\ldots, m$ the exact sequence 
\begin{equation}\label{eq.flat}
	0 \to K_{i-1} \to M_{i-1}  \xrightarrow[\text {}]{F_i} M_{i-1} \to M_i \to 0,  
\end{equation}
with $K_{i-1}$ the kernel of $F_i$. Then $F_1,\ldots, F_m $ is an $M$-regular sequence iff $K_{i-1}=0$ for $i=1,...,m$. 
The sequence  (\ref{eq.flat}) stays exact if we localize 
it at a prime ideal $\fn$ of $R$. Now $K_{i-1}=0$ iff
$(K_{i-1})_\fn=0$ for all prime (equivalently maximal) ideals $\fn$ in $R$ and, by \cite[Theorem 7.1]{Mat86},   $M$ is $A$-flat iff $M_\fn$ is $A_{A\cap \fn}$-flat for all prime (equivalently maximal) ideals $\fn$ in $R$.
Hence (a) equivalent to:

($\alpha$) $M_\fn/ \langle F_1,..., F_m\rangle M_\fn$  is flat over $A_{A\cap \fn}$ and $F_1,\ldots,F_m$  is an $M_\fn$-regular sequence for all prime (equivalently maximal) ideals $\fn$ of $R$.
 
\noindent 
If we fix a prime ideal $\fn$ of $R$ and set $\fp = A\cap \fn$ then, by successive application of  \cite[Theorem 22.5]{Mat86}
applied to $A_\fp \to R_\fn$  and  starting with $F_m$, 
statement ($\alpha$)
is equivalent to:
 
($\beta$) $(M_i)_\fn$ is flat over $A_\fp$ for $i=0,...,m$, 
and  $F_1(\fp),\ldots,F_m(\fp)$  is an $M(\fp)$-regular sequence. 
 
 \noindent If $\fn$ runs through all maximal ideals of 
$\Supp_R(M/IM)$, then $\fp=A\cap \fn$ runs through all prime ideals of 
$\Supp_A(M)$, since $\Supp_R(M/IM)\to \Supp_A(M/I)$ is surjective (see Remark \ref{rm.sem0}). This proves the equivalence of (a) and (b).\\

(ii)  Let us first show that (i)(a) implies (ii)(a). We have
$L^{\wedge} = L\otimes_R R^{\wedge}$ for any finite $R$-module $L$ by Corollary \ref{cor.cpl} (iv). Since $R^{\wedge}$ is flat over $R$, the flatness of $M^\wedge/ \langle F_1,..., F_m\rangle M^\wedge$  follows hence from the flatness of
$M/ \langle F_1,..., F_m\rangle M$.
The $I$-adic completion of sequence (\ref{eq.flat}),
\begin{equation}\label{eq.flat1}
	0 \to K_{i-1}^\wedge \to M_{i-1}^\wedge  \xrightarrow[\text {}]{F_i}  M_{i-1}^\wedge \to  M_i^\wedge \to 0,  
\end{equation}
is an exact sequence of $R^\wedge$-modules and since
 $K^\wedge_{i-1} =0$ if $K_{i-1} =0$, it follows that
 $F_1,\ldots, F_m $ is an $M^\wedge$-regular sequence.
The equivalence of (a) and (b) in (ii) follows by applying (i) to $M^\wedge$.

To show (b) $\Rightarrow$ (c), we use Proposition \ref{prop.pctp}.1. and get $\hat M(\fp) = M(\fp)^\wedge =
(M^\wedge(\fp))^\wedge = (M^\wedge(\fp))\otimes_{R(\fp)} \hat R(\fp)$.
Since $\hat R(\fp)$ is flat over $R(\fp)$,
it follows that $\hat F_1(\fp),\ldots,\hat F_m(\fp)$  is $\hat M(\fp)$-regular. \\

(iii) We have to show, that for $i=1,...,m$  the kernel of the (by $F_i$ induced) map
$$\hat F_i(\fq):\hat M_{i-1}(\fq) \to \hat M_{i-1}(\fq)$$
vanishes and that $\hat M_m(\fq)\ne 0$
for $\fq$ in some  neighbourhood $U$ of  the given $\fp\in \Supp_A(M/I)$. \\

 Lemma \ref{lem.open}.2. and Remark \ref{rm.sem0}
  say that for any finite $R$-module $M$ and $\fp \in \Spec A$,
$\hat M(\fp) = M(\fp)^\wedge \ne 0$ iff 
$\fp \in \Supp_A (M/IM)$ and that the natural map ($\fn \mapsto \fn \cap A$), 
 $$\phi: \Supp_R (M/IM) \to \Supp_A (M/IM)
 = \{\fp \in \Spec A \mid \hat M(\fp) \ne 0\},$$ 
is finite and surjective. Moreover, if $\fn \in \Spec R$, then
\begin{equation}\label{eq.5}
\hat M(\fp)_\fn \ne 0 \iff \fn \in \Supp_R(M)\cap V(I),
\ \fp = \fn \cap A.
\end{equation}

Since $F_1,...,F_m \in I$ it follows that
$\Supp_R (M/IM)
=  \Supp_R (M) \cap V(I)= \Supp_R (M)\cap V(I+\langle F_1,...,F_i\rangle)= \Supp_R(M_i/IM_i)$ 
for $i=1,...,m$.
In particular $\hat M(\fp) \ne 0 \iff \hat M_i(\fp) \ne 0$
for one (equivalently for all) $i=1,...,m$. \\

By Proposition \ref{prop.sem}.4. (and Notation \ref{rm.notation}), 
$\hat M_{i-1}(\fp) =  \oplus_{j=1}^{r_\fp} \hat M_{i-1}(\fp)_{\fm_{\fp,j}}$
and, accordingly,  the map
 $\hat F_i(\fp): \hat M_{i-1}(\fp) \to \hat M_{i-1}(\fp) $ splits as a direct sum
$$\hat F_i(\fp):\oplus_{j=1}^{r_\fp} \hat M_{i-1}(\fp)_{\fm_{\fp,j}}\xrightarrow[\text {}]{\oplus_j \hat F_{i,j}(\fp)}  \oplus_{j=1}^{r_\fp} \hat M_{i-1}(\fp)_{\fm_{\fp,j}}.$$
By assumption the $\hat F_i(\fp)$ and hence the maps $\hat F_{i,j}(\fp) : \hat M_{i-1}(\fp)_{\fm_{\fp,j}} \to \hat M_{i-1}(\fp)_{\fm_{\fp,j}}$ are injective and 
$\hat M_{m}(\fp)_{\fm_{\fp,j}}\ne 0$ for some $j=j_0$. 
Furthermore 
$$\hat M_{i-1}(\fp)_{\fm_{\fp,j}} \cong
(M_{i-1}(\fp)_{\fn_{\fp,j}})^{\wedge,{\fn_{\fp,j}}}
= M_{i-1}(\fp)_{\fn_{\fp,j}}\otimes_{R_{\fn_{\fp,j}}} (R_{\fn_{\fp,j}})^{\wedge,{\fn_{\fp,j}}}.$$
As $(R_{\fn_{\fp,j}})^{\wedge,{\fn_{\fp,j}}}$ 
is faithfully flat over the local ring $R_{\fn_{\fp,j}}$, it follows that the map 
$$F_{i,j}(\fp): M_{i-1}(\fp)_{\fn_{\fp,j}} \to M_{i-1}(\fp)_{\fn_{\fp,j}}$$
 is injective for all $i=1,...,m, \ j=1,...,r_\fp$
 and $M_{m}(\fp)_{\fm_{\fp,j_0}}\ne 0$. 
 
 We  apply now repeatedly \cite[Theorem 22.5, (2) $\Rightarrow$ (1)]{Mat86},  to the morphism of local rings
$(A_\fp,k(\fp)) \to (R_{\fn_{\fp,j}},k({\fn_{\fp,j}})$ 
and to the map
$$F_{i,\fp,j}: (M_{i-1})_{\fn_{\fp,j}} \to (M_{i-1})_{\fn_{\fp,j}},$$
which satisfies
$$F_{i,j}(\fp) = F_{i,\fp,j}\otimes_Ak(\fp): (M_{i-1})_{\fn_{\fp,j}}\otimes_Ak(\fp)=M_{i-1}(\fp)_{\fn_{\fp,j}}\to M_{i-1}(\fp)_{\fn_{\fp,j}}.$$
 Starting with $i=1$ ($M$ is  $A$-flat  by assumption, hence
 $M_{\fn_{\fp,j}}$ is $A_\fp$-flat)  we get
that the kernel 
$(K_{i-1})_{\fn_{\fp,j}}$ of $F_{i,\fp,j}$
vanishes 
and that 
$(M_{i-1})_{\fn_{\fp,j}}/F_{i,j}(M_{i-1})_{\fn_{\fp,j}}=(M_{i})_{\fn_{\fp,j}}$ is flat over $A_\fp$ for all $i,j$.\\

The support of the $R$-module $K_{i-1}$ is closed and hence
 $(K_{i-1})_{\fn}=0$  for $\fn$ in an open neighbourhood   $U_{i-1,j}$  of $\fn_{\fp,j}$ in $\Spec R$.
 That is,
$F_{i,j}: (M_{i-1})_{\fn} \to (M_{i-1})_{\fn}$ is injective for
all $\fn \in U_{i-1,j}$, in particular for all $\fn \in V_{i-1,j}:=U_{i-1,j} \cap \Supp_R (M/IM)$. 
Since $M_m$ is $A$-flat
by assumption,
$(M_{m})_\fn$ is $A_{\fn \cap A}$-flat for all $\fn \in \Supp_R(M_m) \supset \Supp_R (M/IM)$. 
It follows  from  \cite[Corollary to Theorem 22.5, (1) $\Rightarrow$ (2)]{Mat86} that 
$$F_{i,j}(\fq): (M_{i-1})_{\fn}\otimes_A k(\fq)= M_{i-1}(\fq)_{\fn}\to M_{i-1}(\fq)_{\fn},  
\ \fq = \fn\cap A,$$
 is injective for all $\fn \in V_{i-1,j}$.
 
 Set $V_{i-1}:=\bigcup_{j=1}^{r_\fp}V_{i-1,j}$. 
 Then $V_{i-1}$ is open in $\Supp_R(M/IM)$
 and
 disjoint from $\Supp_R(K_{i-1})$. 
Moreover, $\phi(V_{i-1})$ is disjoint from $\phi(\Supp_R(K_{i-1}))$. To see this, assume that there exists a
 $\fq \in  \phi(\Supp_R(K_{i-1}))$ and an   $\fn \in V_{i-1}$ such that $\phi(\fn) = \fq$. By (\ref{eq.5})
 $\hat K_{i-1}(\fq)_\fn = ((K_{i-1}(\fq))_\fn)^\wedge = 
 ((K_{i-1})_\fn(\fq))^\wedge  \ne 0$ iff 
 $\fn \in \Supp_R (K_{i-1}) \cap V(I)$. This is a contradiction since $(K_{i-1})_\fn =0$ if $\fn \in V_{i-1}$.\\
  
It follows that
  $F_i(\fq): M_{i-1}(\fq) \to  M_{i-1}(\fq)$  and,  since $\hat R(\fq)$ is  flat over  $R(\fq)$,
 $\hat F_i(\fq): \hat M_{i-1}(\fq)=M_{i-1}(\fq)\otimes_{R(\fq)}\hat R(\fq)
   \to \hat M_{i-1}(\fq)$
 is injective for $\fq \in U:=\phi (V_{i-1})$ and $i=1,...,m$. Since $\phi$ is open by assumption, $U$ is open in $\Spec A$. Furthermore, if $\fn \in V_{i-1}$, then
 $\fn \in \Supp_R(M/IM) = \Supp_R (M_m)\cap V(I)$.
 Hence $\hat M_m(\fq) \ne 0$  by Remark \ref{rm.sem0}
 for all $\fq \in U$. This proves (iii).
\end{proof}
\bigskip

\noindent {\bf Families of hypersurface singularities}
\medskip

We investigate in this part the behavior of the Milnor number and the Tjurina number in families of hypersurface singularities. We start with the definition for a single hypersurface singularity.

Let $K$ be a field and $K[[x]]$ the power series ring over $K$ with   maximal ideal $\x = \langle x_1,...,x_n\rangle$. For $f \in K[[x]]$ we denote by
$j(f) := \langle {\partial f}/{\partial x_1}, \ldots, {\partial f}/{\partial x_n}\rangle$ the {\em Jacobian ideal} of $f$,
generated in $K[[x]]$ by the partial derivatives of $f$.

\begin{definition}\label{def.mutau}
For $f\in \x K[[x]]$ we define
\begin{enumerate}
\item[(i)] the {\em Milnor number} of $f$ (w.r.t. $K$), 
$$ \mu_K (f):= \dim_K K[[x]]/j(f)$$

\item[(ii)] and  the {\em Tjurina number} of $f$ (w.r.t. $K$), 
$$ \tau_K (f):= \dim_K  K[[x]]/\langle f,j(f)\rangle.$$

\item[(iii)] We say that $f$ resp. $K[[x]]/\langle f\rangle$   has an {\em isolated singularity} (w.r.t. $K$) if $\mu_K(f) < \infty$
resp.  $\tau_K(f) < \infty$. 
\end{enumerate}
\end{definition}

\begin{remark}\label{rem.mutau} {\em
To shorten notations we set $R:=K[[x]]$ and $A:=K[[t]]$.
\begin{enumerate}
\item [(1)]   We have always  $\tau_K(f) \leq \mu_K(f)$  and
  $\tau_K(f) = \mu_K(f)  \iff f \in j(f)$. Moroever,
$\mu_K(f) < \infty$
resp.  $\tau_K(f) < \infty$ iff $\x^k \subset j(f)$ resp. $\x^k \subset \langle f,j(f)\rangle$ for some $k\ge 0$. 
Let us call 
$$V(j(f))=  \{\fn \in \Spec K[[x]] \mid {\partial f}/{\partial x_i} \in \fn, \ i=1,...,n \}$$
the {\em singular locus of the power series} $f$.
By Lemma \ref{lem.slocus} the {\em singular locus of the local algebra}  $K[[x]]/\langle f \rangle)$ is 
$$\Sing (K[[x]]/\langle f \rangle) = V(\langle f,j(f)\rangle).$$ 
Hence $\mu_K(f) < \infty$ 
(resp. $\tau_K(f) < \infty$) iff  $\Sing (f)$ (resp. $\Sing (K[[x]]/\langle f \rangle)$ is either $\{\x\}$ or $\emptyset$. 

\item [(2)] Consider the map 
$\phi: A=K[[t]] \to R=K[[x]]$, $\phi(t) = f$, 
and the induced map
$\Spec \phi: \Spec R \to \Spec A$.
For $\fn \in \Spec R$ set  $\fp =\fp_\fn =(\Spec \phi) (\fn)$, i.e. $\fp =A \cap \fn= \phi^{-1}(\fn)$.
Since $\phi$ is flat, its non-smooth locus of $\phi$ satisfies (cf. Remark \ref{rm.sopen})
\begin{center}
$\NSm(\phi) = \{\fn \in \Spec R \mid R(\fp)_\fn \text{ is not geometrically regular over } k(\fp) \}.$
\end{center}

\item [(3)] Let $\chara(K) = 0$. 
Then the  singular locus of $f$
 coincides with the singular locus of $K[[x]]/\langle f \rangle$ and with the non-smooth locus of $\phi$ i.e. 
$$V(j(f)) = V(\langle f,j(f)\rangle) = \NSm(\phi).$$

\noindent In fact,  in characteristic 0 we  have $V(j(f)) \subset V(f)$.
This follows from \cite[Theorem 7.1.5]{SH06} which implies that $f$ is contained in the integral closure and hence in the radical of $j(f)$. Therefore $f^k \in j(f)$ for some $k$ and thus 
$V(j(f)) \subset V(f)$.
The Brian\c{c}on-Skoda theorem \cite[Theorem 13.3.3]{SH06} implies that $j(f)^n$ is contained in the integral closure of $j(f)$, hence even $f^n \in j(f)$. This proves the first equality, the second follows from Lemma \ref{lem.smphi}.\\
It is proved in \cite[Theorem 2]{BGM12} that
 $\mu(f) < \infty \Leftrightarrow \tau(f) < \infty$,  but the example $f=x^3+y^5$ shows that this is not true in characteristic 3 and 5.
 
\item [(4)] Let $\chara(K) = p >0$. Then 
$$V(\langle f,j(f)\rangle) \subset  \NSm(\phi) \text{ and }V(\langle f,j(f)\rangle)  \subset  V(j(f)),$$
 and in general the inclusions are strict. The first inclusion follows from Lemma \ref{lem.smphi} with equality if $K$ is perfect.
The second inclusion is clear. 
For $f=x^3+y^5$, $p=3,5$, we have $V(\langle f,j(f)\rangle) = V(x,y)$, which is strictly contained in $V(j(f)) = V(y)$ (even if $K$ is perfect).
\end{enumerate}
}
\end{remark}

\medskip

Let us consider $\phi: A=K[[t]] \to R=K[[x]]$, $ t \mapsto f$, in more detail. We are interested in the relation between the 
non-smooth locus of $\phi$ and the singular locus of $f$ and 
of $R/\langle f \rangle$ in arbitrary characteristic.

 $\Spec A$ has only two points, the closed point  $\fp = \langle t \rangle$ and the generic point $\fp = \langle 0 \rangle$ with residue fields $k(\langle t \rangle) = K$ and
$k(\langle 0 \rangle) = \Quot (A) = K((t)).$
We call $R(\langle t \rangle) = K[[x]]/\langle f \rangle$ the {\em special fiber} 
and 
$R(\langle 0 \rangle) = K[[x]] \otimes_{K[[t]]} K((t))$ the
 {\em generic fiber} of $\phi$.
We note that 
$(\Spec \phi) ^{-1}(\langle t\rangle) =V(f) \text{ and }
(\Spec \phi) ^{-1}(\langle 0\rangle) = D(f) = \Spec R_f.$
Let $\fn \in \Spec R$. Then 
$$(\Spec \phi)(\fn) = \langle 0\rangle \iff
f\notin \fn \text{ and } (\Spec \phi)(\fn) = \langle t\rangle \iff   f \in \fn.$$
If $f\notin \fn$ then  $R(\langle 0\rangle)_\fn = R_\fn\otimes_{A} \Quot(A)= R_\fn\otimes_{A} K((t))=R_\fn=K[[x]]_\fn $, 
since $t$ acts as $f$, which is a unit in $R_\fn$. On the other hand, if $f \in \fn$ then
$R(\langle t\rangle)_\fn = K[[x]]_\fn/\langle f\rangle$.\\

The following lemma shows that the generic fiber $R(\langle 0\rangle)$  is always contained in the smooth locus $\Sm(\phi)$. The special fiber $R(\langle t\rangle)$ is however more subtle (it may be regular, but not geometrically regular over $k(\langle t \rangle)$.

\begin{lemma} \label{lem.smphi}
Let $K$ be a field, 
$f\in K[[x]]$ and
$\phi: A=K[[t]] \to R=K[[x]]$, $\phi(t) = f$.
Then $ V(\langle f,j(f)\rangle)$ is contained in  the 
 non-smooth locus $\NSm(\phi)$ of $\phi$, 
 with equality if $K$ is perfect.\\
More precisely for  $\fn \in \Spec R$ the following holds. 

\begin{enumerate}
\item  If  $\Spec \phi)(\fn) = \langle 0\rangle$
 then $R(\langle 0\rangle)_\fn$ is geometrically regular over $K$ and hence also over the extension field $k(\langle 0\rangle)$.That is, the generic fiber $R(\langle 0\rangle)$ 
is geometrically regular over 
$k(\langle 0\rangle)$ and hence $\Spec R(\langle 0\rangle) \subset\Sm(\phi)$.

\item 
Let $(\Spec \phi)(\fn) = \langle t \rangle$ and  $R(\langle t \rangle) = R/\langle f\rangle$ 
 the special fiber.
If $\fn\in  V(\langle f,j(f)\rangle)$ 
then  $R(\langle t \rangle)_\fn$ is not geometrically regular over $k(\langle t \rangle) = K$ and hence $\fn \in \NSm(\phi)$. \\
If $\fn\notin  V(\langle f,j(f)\rangle)$ then $R_\fn/\langle f\rangle$ is geometrically regular over $k(\fn)$
(but not necessarily over $k(\langle t \rangle) = K$  if $\fn \ne \x$;  if $\fn = \x$ then $k(\fn) =K$). If $K$ is perfect, then $R(\langle t\rangle)_\fn$ is  geometrically regular over $K$ and hence $\fn \in \Sm(\phi)$.
\end {enumerate}
\end{lemma}

\begin{proof}
1. That  $R(\langle 0\rangle)_\fn=K[[x]]_\fn$ is
geometrically regular over $K$ follows from Lemma \ref{lem.smooth}, Lemma \ref{lem.jac} and Proposition  \ref{prop.greg} and hence $\Spec R(\langle 0\rangle) \subset
\Sm(\phi)$.

2. By 1. $\NSm(\phi) \subset V(f)$. Let $\fn \in \Spec R$ and $\fn \in V(f) \cap \Sm(\phi)$. We want to show that $\fn \notin V(j(f))$, i.e. $j(f)\not\subset \fn$. We note that $j(f) \subset \fn$ implies that $j(f)R_\fn \subset \fn R_\fn$
 and hence that $j(f)(R_\fn)^{\wedge, \fn}$ is contained in the maximal ideal of $(R_\fn)^{\wedge, \fn}$.

Since $\fn \in \Sm(\phi)$,  we have
$R(\langle t\rangle)_\fn = R_\fn/\langle f\rangle$
is geometrically regular over  $k(\langle t\rangle) =K$ (Remark \ref{rem.mutau} (2)) and hence over $k(\fn)$.
By Proposition \ref{prop.greg} this is equivalent to $(R_\fn)^{\wedge, \fn}/\langle f\rangle$ being geometrically regular over $k(\fn)$.
By Cohen (Theorem \ref{thm.cohen}) $(R_\fn)^{\wedge, \fn} \cong k(\fn)[[y_1,...,y_d]]$, 
$d=n-\dim K[[x]]/\fn$,  and 
$k(\fn)[[y_1,...,y_d]]/\langle f\rangle$ being regular implies
that $j(f) k(\fn)[[y_1,...,y_d]] = k(\fn)[[y_1,...,y_d]]$ by Lemma \ref{lem.jac}. By the above, this implies that $j(f) \not\subset \fn$ and hence $\fn \notin  V(\langle f,j(f)\rangle)$ i.e.
$V(\langle f,j(f)\rangle) \subset \NSm(\phi)$.

Conversely, let $\fn \in V(\langle f\rangle)$, $\fn \notin V(j(f))$. Then $j(f) \not\subset \fn$ and by the argument above $j(f) k(\fn)[[y_1,...,y_d]] = k(\fn)[[y_1,...,y_d]]$, implying that
$(R_\fn/\langle f\rangle)^{\wedge, \fn} \cong  k(\fn)[[y_1,...,y_d]] / \langle f\rangle $ is geometrically regular over $k(\fn)$ (by Lemma \ref{lem.jac}) and hence $R_\fn/\langle f\rangle$ is geometrically regular over $k(\fn)$. If $K$ is perfect, then $k(\fn)$ is separable algebraic over $K$ and  $R_\fn/\langle f\rangle$ is geometrically regular over $K$ (by Lemma \ref{lem.reg.fieldextension}), and hence $\fn \in \Sm(\phi)$.
By Example 
\ref{ex.nsep}, $R_\fn/\langle f\rangle$ need not be geometrically regular over $K$  and   
$\fn \notin \Sm(\phi)$ if $K$ is not perfect.
\end{proof}

\begin{example}\label{ex.nsep}
\em{
Let $K$ be a field of characteristic $p>0$, $K'=K^{\frac{1}{p}}$ and $\alpha \in K' \smallsetminus K$. Consider 
$A=K[[t]] \to R=K[[x,y]]$, $t \mapsto f$, with
 $f=x^p+\alpha y^p$. Since $f$ is irreducible, $\fn = \langle f \rangle  \subset K[[x,y]]$ is a prime ideal and since 
$\fn \cap A = \langle t \rangle$, we get $k(\langle t \rangle) = K$.
The localisation of the special fiber at $\fn$ is $R(\langle t \rangle)_\fn = (K[[x,y]]/ \langle f \rangle)_{\langle f \rangle} = \Quot(K[[y]][x]/\langle f \rangle)=K((y))[x] / \langle f \rangle=K((y)) (\sqrt[p]{\alpha}y)$, which is a field (hence regular) equal to $k(\fn)$.
We get that
$R(\langle t \rangle)_\fn \otimes_K K' = (K((y))\otimes_K K')[x] / \langle (x+ \sqrt[p]{\alpha}y)^p \rangle$ is not reduced
and hence not regular, showing that $R(\langle t \rangle)_\fn$ is not geometrically regular over $K$. 
Note that $k(\fn)=R(\langle t \rangle)_\fn$ is not separable over $K$ (\cite[Chapter 26]{Mat86})
and not algebraic over $K$. }
\end{example}
\medskip

We consider now families of singularities. For the rest of this section we fix the following notations and assumptions.
We assume  that $A$ is a Noetherian ring and set
$$R:=A[x], \ x=(x_1,...,x_n).$$ 
The natural map  $\phi:A\to R$ is smooth, as it is flat ($R$ is even a free $A$-module) and the non-empty fibers $R(\fp)=k(\fp)[x]$ are geometrically regular over $k(\fp)$ (since for every maximal ideal $\fn$ of $k(\fp)$, the completion $(k(\fp)[x]_\fn)^{\wedge,\fn}
\cong k(\fn)[[x]] $ is geometrically regular).\\
A {\em family of hypersurfaces} is given by an $F\in R$ so that the fibers $F(\fp)$ are polynomials in $k(\fp)[x]$ and we will show that the sum of the Milnor resp. Tjurina numbers of $F(\fp)$ is semicontinuous if $\fp$ varies in $\Spec A$.\\
We assume moreover, that an ideal $I\subset  R$ is given with {\em $R/I$ finite over $A$}. We will consider the Milnor and Tjurina number "along $V(I)$", meaning that we count  these numbers only at intersection points of the fibers $\Spec R(\fp)$ with $V(I)$. Two cases are particularly interesting: \\
(a) when $I = \langle F\rangle$ resp. $I = \langle F,j(F)\rangle$. Then we count the Milnor and Tjurina numbers at all points of the fibers.\\
(b) when $I$ defines a section of $\phi$, that is $I = \ker(\sigma)$ where $\sigma: R \to A$ is a section of $\phi$ satisfying $\sigma \circ \phi = id_A$ 
(for instance the trivial section $I=\x$.)

 \begin{remark}\label{rm.facts} {\em
 We recall some facts from the previous chapters in our situation. Let $\fp \in \Spec A$. 
 \begin{enumerate}
\item $R_\fp=A_\fp[x] $, 
 $ R(\fp) = k(\fp)[x] $ and  (by Proposition \ref{prop.pctp}.1.)
$$ \hat R(\fp) = (k(\fp)[x])^{\wedge,Ik(\fp)[x]}.$$
\item  $\hat R(\fp)/I \hat R(\fp)$  is finite dimensional over $k(\fp)$ and a semi-local ring  with  maximal ideals 
$\bar\fm_{\fp,1},...,\bar\fm_{\fp,r_fp}$ (this is the starting point in Proposition \ref{prop.sem}).\\
$\hat R(\fp) \ne 0 \iff  \fp \in \Supp_A(R/I) \iff I_\fp \subsetneqq R_\fp$ and the map $\Spec \phi: V(I) \to \Supp_A(R/I)$, $\fn \mapsto \fn\cap A$, is surjective and finite (Lemma \ref{lem.open}). Hence 
$$(\Spec \phi)^{-1}(\fp)\cap V(I) = \{\tilde \fn_{\fp,1},...,\tilde \fn_{\fp,r_\fp}\}$$
for finitely many prime ideals 
$ \tilde \fn_{\fp,1}, ...,\tilde \fn_{\fp,r_\fp}$  in $R$. 
By  Remark \ref{rm.sem0} we have
$\tilde \fn_{\fp,_i} =R\cap \bar\fm_{\fp,i}$ with $\bar\fm_{\fp,i}$ the maximal ideals of the semi-local ring 
$\hat R(\fp)/I \hat R(\fp)$
(Proposition \ref{prop.sem}.1.).  
\item
The ideals $\fn_{\fp,i} = \tilde \fn_{\fp,i} R_\fp$ resp. 
$\fm_{\fp,i} = \tilde \fn_{\fp,i} \hat R(\fp)$ are maximal ideals in $R_\fp$ resp. in $\hat R(\fp)$. $\hat R(\fp)$ is semi-local and a direct sum
\begin{equation}\label{eq.dirsum}
\hat R(\fp) \cong   \hat R(\fp)_{\fm_{\fp,1}}\oplus\ldots\oplus \hat R(\fp)_{\fm_{\fp,r_\fp}},
\quad \{\fm_{\fp,1},...,\fm_{\fp,r_\fp} \} = V(I \hat R(\fp)),
\end{equation}
with
$\hat R(\fp)_{\fm_{\fp,i}} \cong
(k(\fp)[x]_{\fm_{\fp,i}})^{\wedge,\fm_{\fp,i}}
\cong k(\fm_{\fp,i})[[x]]$ and $k(\fm_{\fp,i})$ finite over $k(\fp)$.
\item
 If $M$ is a finitely generated $R$-module, then  $\hat M(\fp)$
is finite over $k(\fp)$ and
$$
\begin{array}{lll}
\dim_{k(\fp)}  \hat M(\fp)&=
\sum_{i=1}^{r_\fp} \dim_{k(\fp)} \hat M(\fp)_{\fm_{\fp,i}}\\
\smallskip
&=\sum_{i=1}^{r_\fp} [k(\fm_{\fp,i}):k(\fp)]\dim_{k(\fm_{\fp,i})} \hat M(\fp)_{\fm_i}.\\
\end{array}
$$ 
\end{enumerate}
}
\end{remark}
\bigskip

\noindent Now let $F \in A[x]$ be given. We define
$$j(F):=\langle {\partial F}/{\partial x_1}, \ldots, {\partial F}/{\partial x_n}\rangle \subset A[x],$$ 
the relative Jacobian ideal of $F$.
\medskip

Denote by $ \hat F(\fp)$ the image of $F$   in $\hat R(\fp)$ 
and by $\hat F(\fm_{\fp,i})$
the images 
in $\hat R(\fp)_{\fm_{\fp,i}} \xrightarrow {\cong} k(\fm_{\fp,i})[[x]]$.  
Note that $(j(F)(\fp))^\wedge = j(\hat F(\fp)) = \langle \partial{\hat F(\fp)}/\partial {x_1},..., \partial{\hat F({\fp})}/\partial {x_n}\rangle$.
Then, in $\hat R(\fp) \cong \oplus_{i=1}^{r_\fp}  k(\fm_{p,i})[[x]]$,
$$
\begin{array}{lll}
\hat F(\fp) &=& \hat F(\fm_{\fp,1})\oplus ...\oplus \hat F(\fm_{\fp,r_\fp}),\\
j(\hat F(\fp))&=& j(\hat F(\fm_{\fp,1}))\oplus ...\oplus j(\hat F(\fm_{\fp,r_\fp})),
\end{array}
$$
and the Milnor  and the Tjurina number of $\hat F(\fm_{\fp,i})\in k(\fm_{\fp,i})[[x]]$ are defined. We have
$$
\begin{array}{lll}
\mu_{k(\fm_{\fp,i})}(\hat F(\fm_{\fp,i}))&=&\dim_{k(\fm_{\fp,i})} 
 k(\fm_{\fp,i})[[x]]/ j(\hat F(\fm_{\fp,i})),\\
 \medskip
 
 \tau_{k(\fm_{\fp,i})}(\hat F(\fm_{\fp,i}))&=&\dim_{k(\fm_{\fp,i})} 
 k(\fm_{\fp,i})[[x]]/ \langle \hat F(\fm_{\fp,i}), j(\hat F(\fm_{\fp,i}))\rangle.
\end{array}
$$

\begin{definition}\label{def.mutau2}
Let $F\in R=A[x]$, $I\subset R$ an ideal with $R/I$ finite over $A$.
For $\fp \in \Spec A$ let $\hat R(\fp) = R(\fp)^{\wedge, I}$ be the $IR(\fp)$-adic completion of the fiber $R(\fp)$.
Define
\begin{enumerate}
\item the {\em Milnor number of $\hat F(\fp)$ along $V(I)$}  by
$$ \mu_{k(\fp)}^I (\hat F(\fp)) :=  \dim_{k(\fp)}  \hat R(\fp)/j(\hat F(\fp)),$$ 
\item the {\em Tjurina number of $\hat F(\fp)$ along $V(I)$} by
$$ \tau_{k(\fp)}^I(\hat F(\fp)) :=  \dim_{k(\fp)}  \hat R(\fp)/\langle \hat F(\fp),j(\hat F(\fp))\rangle.$$ 

\item If $I = j(F)$ we call  $\mu_{k(\fp)} (\hat F(\fp)):=\mu^{j( F)}_{k(\fp)} (\hat F(\fp))$ the {\em total Milnor number} of $\hat F(\fp)$ and if  $I = \langle F,j(F)\rangle$, we call  $\tau_{k(\fp)} (\hat F(\fp)):=\tau^{\langle F,j(F)\rangle}_{k(\fp)} (\hat F(\fp))$ the {\em total Tjurina number} of $\hat F(\fp)$.
\end{enumerate}
\end{definition}
\medskip

We say "along $V(I)$" since in the formulas
\begin{equation}\label{eq.mutau}
\begin{array}{llll}
\mu^I_{k(\fp)} (\hat F(\fp)) &=&
 \sum_{i=1}^r [k(\fm_{\fp,i}):k(\fp)] \mu_{k(\fm_{\fp,i})}(\hat F(\fm_{\fp,i})),\\
 \tau^I_{k(\fp)} (\hat F(\fp)) &=&
 \sum_{i=1}^r [k(\fm_{\fp,i}):k(\fp)] \tau_{k(\fm_{\fp,i})}(\hat F(\fm_{\fp,i})),
 \end{array}
\end{equation}
the sum is taken only over the points $\fm_{\fp,i} \in V(I\hat R(\fp))$, corresponding to points in  $V(I)$ lying over $\fp$. 

If $I=j(F)$, then $V(I\hat R(\fp)) = \Sing (\hat F(\fp))$ as defined in Remark \ref{rem.mutau} and the sum is taken over all points $\fn \in  \Sing (\hat F(\fp))$; if $\fn \notin \Sing (F(\fp))$ then $\mu_{k(\fp)}(F(\fp))=0$.
Similarly, if $I=\langle F,j(F)\rangle$, then the sum is taken over {\em all }
$\fn \in \Sing (\hat R(\fp)/\hat F(\fp))$ by Lemma \ref{lem.slocus} and 
 $\tau_{k(\fp)}(\hat F(\fp))=0$ if $\fn \notin \Sing (\hat R(\fp)/\hat F(\fp))$.
That's why we say "total Milnor number" and "total Tjurina number".
 \bigskip

We deduce now the semicontinuity of $ \mu^I_{k(\fp)}(\hat F(\fp))$ and $\tau^I_{k(\fp)}(\hat F(\fp)).$

\begin{proposition} \label{prop.mutau}
With the notations above, let $R=A[x]$, $F\in \x R$ and  $I\subset R$ an ideal such that $R/I$ finite over $A$. 
\begin{enumerate}
\item
$\mu^I_{k(\fp)} (\hat F(\fp))$ and $\tau^I_{k(\fp)} (\hat F(\fp))$ along $V(I)$  are semicontinuous on $\Spec A$. This holds
in particular for the total Milnor and Tjurina number.
\item  Let  $R/j(F)$ be finite over $A$.  Then $R/j(F)$ is flat over $A$ and the total Milnor number $\mu_{k(\fp)} (\hat F(\fp))$ is locally constant on $\Spec A$.
 \end{enumerate}
 
 \end{proposition}

\begin{proof} 
Set $M=R/ j(F)$ resp. $M=R/ \langle F, j(F) \rangle$. 
Then $\hat M(\fp)= \hat R(\fp)/j(\hat F(\fp))$ resp.
$\hat M(\fp)= \hat R(\fp)/\langle \hat F(\fp),j(\hat F(\fp))\rangle$
with $\hat {}$ denoting the $I$-adic completion. 

1. The semicontinuity of $\mu^I_{k(\fp)}(\hat F(\fp))$ resp.
$\tau^I_{k(\fp)}(\hat F(\fp))$
follows by applying Theorem \ref{thm.scont}.1. to $M$.

2. Since $R/j(F)$ is finite over $A$, for each $\fp \in \Spec A$ the fiber $R(\fp)/j(F(\fp)$ is finite over $k(\fp)$; here $F(\fp)$ denotes also the image of $F$ in $R(\fp) = k(\fp)[x_1,...,x_n]$. Since $R(\fp)$ is Cohen-Macaulay of dimension $n$ and $R(\fp)/j(F(\fp))$ has dimension 0,
the sequence 
$\partial{F(\fp)}/\partial {x_1},..., \partial{F({\fp})}/\partial {x_n}$ is a $k(\fp)[x]$-regular sequence. Since $R$ is flat over $A$ it follows from \cite {Mat86}, Theorem 22.5 - 22.6 that $R/j(F)$ is flat over $A$.
Hence Lemma \ref{lem.free} implies that 
$\dim_{k(\fp)} R(\fp)/j(F(\fp))$ is locally constant on $\Spec A$. Proposition \ref{prop.Afin} gives
$\dim_{k(\fp)} R(\fp)/j(F(\fp))=\dim_{k(\fp)} \hat R(\fp)/j(\hat F(\fp)) =\mu_{k(\fp)}(\hat F(\fp))$,  since 
$\Supp_R(R/j(F)) = V(j(F))$. This implies statement 2..
\end{proof}

\begin{example}\label{ex.mutau} {\em
 Let $K$ be a field, $A=K[t]$ and $R=A[x,y]$. A prime ideal of $A$ is either a maximal ideal of the form $\fp =\langle p(t) \rangle$, $p\in K[t]$ an irreducible (monic) polynomial, or the generic point $\fp =\langle 0 \rangle$. 
 
 If $\fp$ is maximal, then $k(\fp) = K[t]/\langle p \rangle$ is a finite field extension of  $K$ of degree 
 $\deg(p)$. With $\bar t $ the residue class of $t$ in $K[t]/\langle p \rangle$, we have
$k(\fp) \cong K(\alpha)$, $\alpha =\bar t$ a root of $p$ and
 $R(\fp) = K(\alpha)[x,y]$.
  If  $\fp =\langle 0 \rangle$ we get $k(\fp) = K(t)$  and 
  $R(\fp) = K(t)[x,y]$.
 \medskip
 
For $F\in A[x,y]$ we have
$j(F)(\fp) =j(F)\cdot k(\fp)[x,y] \in R(\fp)$. 
We compute the Milnor resp. Tjurina number of $F(\fp)$ 
according to the formulas \eqref{eq.mutau} and we determine $\hat R(\fp)$. We also show that  $\mu_K (F(\fp))$ may not be semicontinuous on  $\Spec A$, not even on the set of closed points (in contrast to $\mu_{k(\fp)} (F(\fp))$).

 \begin{enumerate}
\item [(1)] Let $F=x^2+y^2-2tx$ and $\chara (K) \ne 2$. Then $j(F) = \langle x-t, y\rangle$, 
$R(\fp)/j(F)(\fp)=k(\fp)[x,y]/ \langle x-t, y\rangle  \cong k(\fp)[x,y]/ \langle x, y\rangle)=k(\fp)$ and hence
$\mu_{k(\fp)} (F(\fp)) = 1$ for all $\fp \in \Spec A$.
Moreover,
$$\hat R(\fp) = (k(\fp)[x,y])^{\wedge,\langle x-t, y\rangle}  \ \cong  k(\fp)[x,y])^{\wedge,\langle x, y\rangle}  = k(\fp)[[x,y]]$$ 
is a local ring.
With the notations of Remark \ref{rm.facts}, 
we obtain $r=1$, $\bar \fm_{\fp,1} = \langle 0\rangle$, $\fm_{\fp,1} = \langle x-t, y\rangle \subset
\hat R(\fp) $
 and $k(\fp) = k(\fm_{\fp,1})$.\\

Since $\langle F , j(F) \rangle= \langle x-t, y, t^2 \rangle$ we get  $\tau_{k(\fp)} (F(\fp)) = 1$  if \ $\fp =\langle t \rangle$
and $\tau_{k(\fp)} (F(\fp)) = 0$ if $\fp \in \Spec A \setminus \{\langle t \rangle\}$ since $t^2$ is a unit in $\hat R(\fp)$.

\item [(2)] Let $F=x^3+y^2-3tx$, $\chara (K) \ne 2,3$. Then $j(F) = \langle x^2-t,y\rangle$ and $\langle F , j(F) \rangle= \langle y,t^2,xt,x^2-t\rangle$. For $\fp =\langle p(t) \rangle$ a maximal ideal and $\alpha$ a root of $p(t)$, we get
$\hat R(\fp)/j(F(\fp)) = K(\alpha)[x,y]/\langle x^2 -\alpha, y\rangle \cong K(\alpha)[x]/\langle x^2 -\alpha \rangle.$ This is either a local or semi-local ring, depending on $\fp$.   \\
Let $\fp$ be maximal.\\
- If $\sqrt{\alpha} \in  K(\alpha)$ then 
$\hat R(\fp)/j(F(\fp))$ is semi-local (a direct sum of two fields) with maximal ideals 
$\bar \fm_{1,2}= \overline{\langle x \pm\sqrt{\alpha},y \rangle}$. Then, since $K(\alpha)[x,y]^{\wedge, \langle x \pm\sqrt{\alpha},y\rangle} \cong  K(\alpha)[[x,y]]$,
we have according to proposition \ref{prop.sem}
$$\hat R(\fp) \cong K(\alpha)[[x,y]] \oplus K(\alpha)[[x,y]].$$
- If $\sqrt{\alpha} \notin  K(\alpha)$ then 
$\hat R(\fp)/j(F(\fp) \cong K(\sqrt{\alpha})$ is field 
and 
$$\hat R(\fp) \cong K(\alpha)[x,y]^{\wedge, \langle x^2 -\alpha,y\rangle} \cong K(\sqrt{\alpha})[[x,y]].$$
Let $\fp =\langle 0 \rangle$. Then 
 $\hat R(\fp)/j(F(\fp))
\cong K(t)[x,y]/\langle x^2-t,y\rangle\cong K(\sqrt{t})$ and hence $\hat R(\fp) \cong K(\sqrt{t})[[x,y]].$\\
We get $\mu_{k(\fp)} (F(\fp)) = 2$ for all $\fp \in \Spec A$
and (as in (1))  $\tau_{k(\fp)} (F(\fp)) = 2$  if \ $\fp =\langle t \rangle$
and $\tau_{k(\fp)} (F(\fp)) = 0$ if $\fp \in \Spec A \setminus \{\langle t \rangle\}$.
\item [(3)] For $\fp=\langle p(t)\rangle$ a maximal ideal we get
$\mu_K (F(\fp)) =\deg(p)\cdot \mu_{k(\fp)} (F(\fp)) <\infty$,
but since $K(t)$ is a infinite extension of $K$ we get $\mu_K (F(\langle 0 \rangle))= \infty$ if $\fp = \langle 0 \rangle$.
Hence $\mu_K(F(\fp))$ is not semicontinuous on $\Spec K[t]$.\\
For example, let $F=x^3+y^2-3tx \in \Q[t,x,y]$. If $\fp =\langle t-c \rangle$, $c \in \Q$, then $\mu_{\Q} (F(\fp)) =2$.
If $\fp_n =\langle t^2-1/n \rangle$ and $n$ not a square, then $\mu_{\Q} (F(\fp_n)) =4$ and $\fp_n$ is arbitrary close to $\langle t\rangle$ for large $n$. This shows that  $\mu_{\Q} (F(\fp))$ is not semicontinuous on the set of closed points of
$\Spec \Q[t]$. In fact, $\mu_{\Q} (F(\fp))$ can be arbitrary big for $\fp$ close to $\langle t\rangle$.
\end{enumerate}
}
\end{example}

To make the connection between the results of  \cite{GP21} and of this paper  we
consider besides families of polynomials also families of power series. In the following Corollary let $R=A[x]$ or $R=A[[x]]$ and let $I = \ker(\sigma: R\to A)$ be  the ideal of a section $\sigma$  of the natural map
$A\to R$. Then $\hat R(\fp)/I \hat R(\fp) = R(\fp)/I R(\fp)=
(R/I)\otimes_A k(\fp) \cong k(\fp)$ and, according to Remark \ref{rm.facts}.3., $\hat R(\fp) \cong k(\fp)[[x]]$. For $F\in R$ we have $F(\fp) \in k(\fp)[[x]]$ and therefore $\mu_{k(\fp)}^I (F(\fp)) =  \dim_{k(\fp)}  k(\fp)[[x]]/j(F(\fp))$ is the {\em Milnor number along the section $\sigma$}, and
similar for the Tjurina number along the section $\sigma$.

The semicontinuity of $\mu$ and $\tau$ along the trivial section in the case $R=A[[x]]$ was proven in \cite[Proposition 3.1]{GP21} under additional assumptions, which can be now be dispensed with. Theorem \ref{thm.scont} implies

\begin{corollary}\label{cor.mutau}
Let $A$ be a Noetherian ring, $R=A[x]$ or $R=A[[x]]$, $F \in R$, and $I\subset R$ the ideal of a section $\sigma$ of $A\to R$.
 For $\fp \in \Spec A$ denote by $\hat F(\fp)$
the image of $F$ in $k(\fp)[[x]]$.
Then the Milnor and the Tjurina number along the section $\sigma$ 
$$
\begin{array}{lll}
\mu_{k(\fp)}  (\hat F(\fp)) &=& \dim_{k(\fp)}  k(\fp)[[x]]/ j(\hat F(\fp))
\text{ and }\\
\tau_{k(\fp)}  (\hat F(\fp)) &=& \dim_{k(\fp)}  k(\fp)[[x]]/ \langle \hat F(\fp), j(\hat F(\fp))\rangle
\end{array}
 $$
are semicontinuous on $\Spec A$.
\end{corollary}

If $\fp$ is a  maximal ideal of $A$, then we can also prove the individual semicontinuity of $\dim_{k(\fm_{\fp,i})} \hat M(\fp)_{\fm_i}$ and hence of $\mu^I_{k(\fm_{\fp,i})}(\hat F(\fm_{\fp,i}))$ and $\tau^I_{k(\fm_{\fp,i})}(\hat F(\fm_{\fp,i}))$:

 Let $\fm_i$ be a maximal ideal of 
$\hat R(\fp)$ and $\tilde \fn_i =\fm_i \cap P$. Then
according to Proposition \ref{prop.sem2} (since $R$ is of finite type and flat over $A$) and Remark \ref{rm.cont} (2),
$\tilde \fn_i$ is a maximal ideal of $R$ and there  exists   
an open neighbourhood $U_i = \Spec Q_i \subset \Spec R$  of $\tilde\fn_i$, with $\tilde \fn_j \notin U_i$ if $j\ne i$ and 
$\fm_i \hat Q_i(\fp)$ is the only maximal ideal of $\hat Q_i(\fp)$. There exists an  open neigbourhood  $V_i = \Spec A_i \subset \Spec A$ of $\fp$ such that the restriction $U_i \to V_i$ of $\Spec R \to \Spec A$, is flat and open and that
$Q_i/IQ_i$ is finite over $A_i$.

\begin{corollary}\label{cor.mutau2}
With the assumptions of Proposition \ref{prop.mutau} 
 and $\fp$ a  {\em maximal ideal} of $A$
the following holds.
For $\fq \in \Spec A_i$ the ring $\hat Q_i(\fq)$ is semi-local with maximal ideals, say, $\fn_{i,1},...,\fn_{i,t}$,
and $\hat M_{Q_i} := M\hat \otimes_R Q_i$ satisfies $\hat M_{Q_i} (\fp) = \hat M(\fp)_{\fm_i}$. 
Then the semicontinuity of
 $\dim_{k(\fm_{\fp,i})} \hat M(\fp)_{\fm_i}$
 means
$$\dim_{k(\fp)} \hat M(\fp)_{\fm_i} \geq
\sum_{j=1}^t \dim_{k(\fq)} \hat M_{Q_i}(\fq)_{\fn_{i,j}}
 \ \text {for all } \fq \in U_i.$$
In particular,
$$ 
\begin{array}{lll}
 \mu^I_{k(\fm_{\fp,i})} (\hat F(\fm_{\fp,i})) & \geq &
\sum_{j=1}^t [k(\fq):k(\fn_{\fq,j})]\mu_{k(\fn_{\fq,j})}(\hat F(\fn_{\fq,j})),\\
\tau^I_{k(\fm_{\fp,i})} (\hat F(\fm_{\fp,i})) &\geq &
\sum_{j=1}^t [k(\fq):k(\fn_{\fq,j})]\tau_{k(\fn_{\fq,j})}(\hat F(\fn_{\fq,j})). 
\end{array}
$$
\end{corollary}
\bigskip

\noindent {\bf Complete intersection singularities}
\medskip

Let $K$ be a field and $J$ an ideal of $K[[x]]$, $x=(x_1,...,x_n)$,  generated by $f_1, \ldots, f_m$. Recall that the ring
 $K[[x]]/J$ or the ideal $J$ is called a {\it complete intersection} if $\dim(K[[x]]/J) = n - \mng(J)$ with $\mng(J)$ the minimal number of generators of $J \subset K[[x]]$. The complete intersection $K[[x]]/J$  is called an {\it isolated complete intersection singularity (ICIS)} if it has an isolated singularity (cf. Definition \ref{def.iso}).
 

We first consider an arbitrary ideal of $K[[x]]$,  generated by $f_1, \ldots, f_m$.  Set $f:=(f_1,...,f_m) \in K [[x]]^m$ and denote by  $\langle {\partial f}/{\partial x_1}, \ldots, {\partial f}/{\partial x_n}\rangle$ the submodule of $K[[x]]^m$, generated by the m-tuples ${\partial f}/{\partial x_i} = ({\partial f_1}/{\partial x_i},\dots, {\partial f_m}/{\partial x_i})$,  $i=1,\ldots, n$.

\begin{definition}\label{def.tau}
Let $f_1, \ldots, f_m$ be generators of $J$. We set
$$
\begin{array}{lll}
T_{J/K} &:= & K[[x]]^m\big/J\cdot K[[x]]^m+\langle {\partial f}/{\partial x_1}, \ldots, {\partial f}/{\partial x_n}\rangle\cdot K[[x]].\\
\tau_K(J) &:= &\dim_K T_{J/K},
\end{array}
$$
and call $T_{J/K}$ the {\em Tjurina module of $J$ w.r.t. $K$} and
$\tau_K(J)$ the {\em Tjurina number of $J$ w.r.t. $K$}.
\end{definition}

\begin{remark}{\em
The notation $T_{J/K}$ is a bit misleading, as $T_{J/K}$ depends not just on $J$ but also on the chosen generators.
However, $T_{J/K}$ is unique up to isomorphism and 
$\tau_K(J)$ is an invariant of $K[[x]]/J$.

To see this,  let $g_1, \ldots, g_k$ be another  set of generators of $J$. W.l.o.g.  let $m \ge k$ and set $g_j=0$ for $j=k+1,...,m$. 
By Proposition \ref{prop.contact} below there exists an invertible matrix $U \in GL_m(K[[x]])$ such that 
$(f_1,...,f_m) = (g_1,...,g_m)\cdot U$. Then $U$ induces an isomorphism between $T_{J/K}$ defined with $f_1,...,f_m$ and $T_{J/K}$ defined with $g_1,...,g_m$.
}
\end{remark}

\begin{proposition}[ \cite{GPh19}, Proposition 4.3]
\label{prop.contact}
	Let $I$ and $J$ be proper ideals of $K[[x]]$ given by the same number of $m$ generators $\{f_1,...,f_m\}$ and
$\{g_1,...,g_m\}$.
	Then the following are equivalent:
\begin{enumerate}
\item The analytic algebras $K[[x]]/I$ and $K[[x]]/J$ are isomorphic.
\item  There is an automorphism $\phi\in Aut(K[[x]])$ such that $\phi(I)=J$.
\item There are  $\phi\in Aut(K[[x]])$ and $U\in GL(m,K[[x]])$ such that $(f_1,...,f_m) = (\phi(g_1),...,\phi(g_m))\cdot U$.
\end{enumerate}
Moreover, if $I=J$, then $\phi$ can be chosen as the identity.
\end{proposition}

For a complete intersection $T_{J/K}$ is concentrated on the singular locus of $K[[x]]/J$ (Definiton \ref{def.sing}, see also Lemma \ref{lem.slocus})
and $\tau_K(J)$ is the usual Tjurina number for an ICIS. It
is finite iff $J$ has an isolated singularity, which follows from \cite[Lemma 3.1]{GPh19}. In fact, we have for any ideal the following Theorem.
\medskip

\begin {theorem}  [\cite{GPh19}, Theorem 4.6] \label{thm.GPh19}
Let $J \subset  K[[x]]$ be a proper ideal. Then the following are equivalent:
\smallskip

(i) $\tau_K(J) <\infty$.	
\smallskip

(ii)  $K[[x]]/J$ is an ICIS. 
\smallskip
\end{theorem}
\medskip

After considering the Tjurina number for an individual singularity,  let us now consider families. We define first the relative version of $T_{J/K}$.
\medskip

Let $R=A[x]$
and   $J \subset R$ an
ideal generated by elements  $F_1, \ldots, F_m \in R$. 
Set $F:=(F_1,...,F_m) \in R^m$ and denote by  $\langle {\partial F}/{\partial x_1}, \ldots, {\partial F}/{\partial x_n}\rangle$ the submodule of $R^m$, generated by the m-tuples ${\partial F}/{\partial x_i} = ({\partial F_1}/{\partial x_i},\dots, {\partial F_m}/{\partial x_i})$,  $i=1,\ldots, n$.
We define the $R$-module
$$T_{J/A}:=  R^m\big/J\cdot R^m+\langle {\partial F}/{\partial x_1}, \ldots, {\partial F}/{\partial x_n}\rangle\cdot R$$
and call it the {\em relative Tjurina module of $R$ over $A$}.
\medskip

Furthermore let $I\subset R$ an ideal with $R/I$  finite over $A$ and $^\wedge$ the $I$-adic completion.
For $\fp \in \Spec A$ we have the fiber
$R(\fp)=k(\fp)[x]$ and the completed fiber $\hat R(\fp)=k(\fp)[x]^\wedge$.
$\hat R(\fp)$ is a semi-local ring with maximal ideals $\fm_{\fp,1},...\fm_{\fp,r_\fp}$ and the direct sum of 
the local rings (see Remark \ref{rm.facts})
$$\hat R(\fp)_{\fm_\fp,i} \cong k(\fm_{\fp,i})[[x]],
 \text{ with } k(\fm_{\fp,i})  \text{ finite over } k(\fp).$$ 
Let $\hat F_i(\fp)$ resp. $\hat J(\fp)$ denote the images of $F_i$ resp. $J$ in $\hat R(\fp)$. 
Then $\hat J(\fp) = \langle \hat F_1(\fp),...,\hat F_m(\fp) \rangle$ and $\hat T_{J(\fp)/k(\fp)}=(T_{J/A}\otimes_Ak(\fp))^\wedge$ satisfies
$$\hat T_{J(\fp)/k(\fp)}= \hat R(\fp)^m\big/J(\fp)\cdot \hat R(\fp)^m+\langle {\partial F(\fp)}/{\partial x_1}, \ldots, {\partial F(\fp)}/{\partial x_n}\rangle\cdot \hat R.$$
It is the direct sum of its localizations at $\fm_{\fp,1},...,\fm_{\fp,r}$.
In particular, 
$$\tau_{k(\fp)}(\hat J(\fp)) = \dim_{k(\fp)} \hat T_{J(\fp)/k(\fp)}
=  \sum_{i=1}^{r_\fp} \dim_{k(\fm_{\fp,i})} 
(\hat T_{J(\fp)/k(\fp)})_{\fm_{\fp,i}},$$
where 
$$\dim_{k(\fm_{\fp,i})} (\hat T_{J(\fp)/k(\fp)})_{\fm_{\fp,i}}
=\tau_{k({\fm_{\fp,i}})}(\hat J(\fp)_{\fm_{\fp,i}} )$$ 
is the Tjurina number (as defined above) w.r.t. ${k(\fm_{\fp,i})}$  of $\hat J(\fp)_{\fm_{\fp,i}}$, the ideal generated by the image of the $F_i(\fp)$ in 
$\hat R(\fp)_{\fm_\fp,i} \cong k(\fm_{\fp,i})[[x]]$.

We call $\tau_{k(\fp)}(\hat J(\fp))$ the {\em Tjurina number along $V(I)$ of the ideal 
$\hat J(\fp)=\langle F_1(\fp),...,F_m(\fp)\rangle \subset \hat R(\fp)$ w.r.t. $k(\fp)$}, because  $\tau_{k(\fp)}(\hat J(\fp))$ counts the 
Tjurina numbers at the maximal ideals 
 $\fm_{\fp,i} \in V(I\hat R(\fp))$, corresponding to points in  $V(I)$ lying over $\fp$. 
 If $I = \Ann_R (T_{J/A})$ we call $\tau_{k(\fp)}(\hat J(\fp))$ the {\em total Tjurina number} of $\hat J(\fp)$  w.r.t. $k(\fp)$.
 \medskip

\begin{proposition}\label{prop.TI} 
Let $A$ be Noetherian, $R=A[x]$ and   
$I\subset R$ an ideal such that $R/I$ is finite over $A$. Let $F_1, \ldots, F_m \in R$ and  $J \subset  R$ the ideal generated by  $F_1,\ldots, F_m $.  Let $\fp \in \Spec A$ and $\hat J(\fp) \subset \hat R(\fp)$  the ideal generated by $F_1(\fp),\ldots,F_m(\fp) \in \hat R(\fp)$.
Then $\fp$ has an open neighbourhood $U \subset \Spec A$ such that   
$$\tau_{k(\fp)}({\hat J(\fp)})\geq  \tau_{k(\fq)}(\hat J(\fq)) \text{ for all } \fq \in U.$$ 
\end{proposition}
\begin{proof} The finitely generated $R$-module
$T_{J/A}$ satisfies $\hat T_{J(\fp)/k(\fp)}=(T_{J/A}\otimes_Ak(\fp))^\wedge$. The statement follows from Theorem \ref{thm.scont}.
\end{proof}
\medskip

Proposition \ref{prop.TI} holds for any ideal, in particular for complete intersection ideals, which we treat now.

\begin{proposition} \label{prop.icis}
Let $A$ be a Noetherian ring, $R=A[x]$ and  $I \subset R$ an ideal  such that $ R/I R$ is finite over $A$.
Let $J\subset R$ be an ideal generated by 
 $F_1,\ldots, F_m \in R$, such that $ R/J R$ is $A$-flat.
 For $\fp \in \Spec A$ denote  by $\hat J(\fp) \subset  \hat R(\fp)$  the ideal generated by $\hat F_1(\fp),\ldots,\hat F_m(\fp)$.

\begin{enumerate}
\item  If  $\hat I(\fp)$ is a complete intersection, then $\hat I(\fq)$ is a complete intersection for $\fq$ in an open neighbourhood of $\fp$ in $\Spec A$. 
\item  Assume that $\hat I(\fp)$ is an ICIS.
Then  $\hat I(\fq)$ is an ICIS with Tjurina numbers
$\tau_{k(\fp)}(\hat I(\fp)) \geq \tau_{k(\fq)}(\hat I(\fq))$ for $\fq$ in an open neighbourhood of $\fp$ in $\Spec A$.
\end{enumerate}
\end{proposition}

\begin{proof} 1. Since $\hat I(\fp)$ is a complete intersection, we may assume that $\hat F_1(\fp),\ldots,\hat F_m(\fp)$ is a $\hat R(\fp)$-sequence. By Proposition \ref{prop.reg} (iii), applied to $M=R$ (the assumptions are fullfilled), 
  $\hat F_1(\fq),\ldots,\hat F_m(\fq)$ is a $\hat R(\fq)$-sequence, hence $\hat I(\fq)$ is a complete intersection, for $\fq$ in an open neighbourhood of $\fp$ in $\Spec A$.\\
2. follows from Proposition \ref{prop.TI} since  for $\hat J(\fq) \subset \hat R(\fq)]$ a complete intersection  
$\tau_{k(\fq)}(\hat I(\fq))$ is the Tjurina number of $\hat J(\fq)$.
 Moreover, $\hat I(\fq)$ is am ICIS by Theorem \ref{thm.GPh19}.
\end {proof}

Similarly as for the hypersurface singularities, Corollaries \ref{cor.mutau} and  \ref{cor.mutau2} hold for complete intersection singularities.\\

{\bf Problem:}
Let us mention, that the Milnor number $\mu(X)$ for a complex analytic  ICIS $X$ is defined topologically as the middle Betti number of the Milnor fiber of a smoothing. With this definition it can  be proved topologically
 that the Milnor number is semicontinuous under deformations.
 However, the topological definition cannot be used for arbitrary fields.
 
 On the other hand, there exist algebraic formulas by \cite{Gr75} 
and \cite{Le74} which are alternating sums with semicontinuous terms, but ist is not clear if the sum itself is semicontinuous.
Moreover, in \cite{Gr75} the algebraic formula $\mu(X) = \dim_\C \Omega_X^n/d\Omega_X^{n-1}$ is proven, with 
$(\Omega_X^{\bullet},d)$ the Poincaré complex of $X$ and 
$n = \dim (X)$. However, $\Omega_X^n/d\Omega_X^{n-1}$ is not an $\ko_X$-module and our theory cannot be applied.

It remains an open problem, if the Milnor number of an ICIS over an arbitrary field,  defined by one of the algebraic formulas above, is semicontinuous under deformations.


\end{document}